\documentclass[11pt]{article}

\usepackage[total={6in, 9in}]{geometry}

\usepackage{amsthm}

\usepackage{graphicx}
\usepackage{amsmath}
\usepackage {amssymb}
\usepackage{booktabs}
\usepackage[center]{caption}
\usepackage {verbatim}
\usepackage{epsfig,psfrag}
\usepackage{algorithm}
\usepackage{algorithmic}
\usepackage{url}
\usepackage{float}
\usepackage{mathtools}
\usepackage{graphicx}
\usepackage[caption=false]{subfig}
\usepackage{enumerate}
\usepackage{textgreek}
\usepackage{bbm}
\usepackage{appendix}

\DeclarePairedDelimiter\floor{\lfloor}{\rfloor}

\def\R{{\mathbb R}}

\def \Sy{{\mathbb{S}}}
\def\conv{{\rm conv}}

\def\1_xi{\frac{1}{\xi}}

\def\C{{\mathcal C}}
\def\D{{\mathcal D}}

\def\H{{\mathcal H}}

\def\J{{\mathcal J}}

\def\01{\ensuremath{0\mathord{-}1}}

\newcommand{\st} {{\rm s.t.}}

\newtheorem{proposition}{Proposition}
\newtheorem{theorem}{Theorem}
\newtheorem{example}{Example}

\allowdisplaybreaks

\begin{document}



\title{SDP-quality bounds via convex quadratic relaxations for global optimization of mixed-integer quadratic programs}


\author{Carlos J. Nohra\thanks{Department of Chemical Engineering,
        Carnegie Mellon University
        (\texttt{cnohrakh@andrew.cmu.edu}).}
		\and Arvind U. Raghunathan\thanks{Mitsubishi Electric Research Laboratories
        (\texttt{raghunathan@merl.com})}.        
        \and Nikolaos V. Sahinidis\thanks{School of Chemical and Biomolecular Engineering, Georgia Institute of Technology
        (\texttt{nikos@gatech.edu}).}}

\maketitle

\begin{abstract}
We consider the global optimization of nonconvex mixed-integer quadratic programs with linear equality constraints. In particular, we present a new class of convex quadratic relaxations which are derived via quadratic cuts. To construct these quadratic cuts, we solve a separation problem involving a linear matrix inequality with a special structure that allows the use of specialized solution algorithms. Our quadratic cuts are nonconvex, but define a convex feasible set when  intersected with the equality constraints. We show that our relaxations are an outer-approximation of a semi-infinite convex program which under certain conditions is equivalent to a well-known semidefinite program relaxation. The new relaxations are implemented in the global optimization solver BARON, and tested by conducting numerical experiments on a large collection of problems. Results demonstrate that, for our test problems, these relaxations lead to a significant improvement in the performance of BARON.

\end{abstract}


%
%

\section{Introduction}
\label{intro}

We address the global optimization of problems of the form:
\begin{equation}
\label{problem_statement}
\begin{array}{cl}
 \underset{x \in \R^n}{\text{min}}\; &x^T Q x + q^T x\\
 \st           & A x = b \\
               & x_i \in S_i, \; \forall i \in [n] :=\{1, \dots, n\}
\end{array}
\end{equation}
where $Q \in \R^{n \times n}$ is a symmetric matrix which may be indefinite, $q \in {\R}^n$, $A \in {\R}^{m \times n}$ and $b \in {\R}^m$. For each $i \in [n]$, we assume that $S_i$ is a bounded set given by the union of finitely many closed intervals in $\R$.

The formulation in~\eqref{problem_statement} subsumes many classes of problems such as nonconvex quadratic programs (QPs) and mixed-integer quadratic programs (MIQPs), which  arise in applications including facility location and quadratic assignment~\cite{kb:57}, molecular conformation~\cite{pr:94} and max-cut problems~\cite{gw95}.

State-of-the-art global optimization solvers rely on branch-and-bound algorithms to solve~\eqref{problem_statement} to global optimality. The efficiency of these algorithms depends to a large extent on the tightness and the computational cost of the relaxations solved for lower bounding. Commonly used relaxations for~\eqref{problem_statement} are polyhedral relaxations constructed via factorable programming~\cite{mc76,sw01,ts:comp:04} and reformulation-linearization techniques (RLT)~\cite{sa99}, semi-definite programming (SDP) relaxations~\cite{bst:11:mp,bw:13,s:87}, and convex quadratic relaxations which can be obtained via separable programming~\cite{pgr:87}, d.c. programming~\cite{t:95}, or quadratic convex reformulation methods~\cite{bel:13}.

In a recent paper~\cite{nrs:20}, we derived convex quadratic relaxations of~\eqref{problem_statement} by convexifing the objective function through uniform diagonal perturbations of $Q$. These perturbations were constructed by solving eigenvalue and generalized eigenvalue problems involving $Q$ and $A$. Through numerical experiments, we demonstrated that these relaxations are not only inexpensive to solve, but can also provide very tight bounds, significantly improving the performance of branch-and-bound algorithms. Motivated by these results, in the current paper, we consider a related class of convex quadratic relaxations. In particular, we investigate quadratically constrained programming (QCP) relaxations for~\eqref{problem_statement}. These relaxations are derived via quadratic cuts obtained from nonuniform diagonal perturbations of $Q$. We show that these relaxations are at least as tight as the spectral relaxations in~\cite{nrs:20}, and provide a very good approximation of the bounds given by certain SDP relaxations of~\eqref{problem_statement}.

An SDP relaxation for~\eqref{problem_statement} obtained after adding RLT inequalities was considered by Saxena et al.~\cite{sbl:11}, who proposed a procedure to project the feasible region of this relaxation onto the space of the original variables. This projection relies on convex quadratic cuts derived from an SDP separation program which is solved by applying a sub-gradient-based algorithm. Even though the relaxations generated through this approach are nearly as tight as the original SDP formulation, the separation problem used to derive the quadratic cuts can be very expensive to solve.

Dong~\cite{d:16} used a different SDP relaxation for~\eqref{problem_statement} which only includes the diagonal RLT inequalities and showed that this SDP is equivalent to a particular semi-infinite program. This semi-infinite formulation served as a motivation to construct convex QCP relaxations for~\eqref{problem_statement} via convex quadratic cuts derived from a semidefinite separation problem. To ensure that the solution of the separation problem is finitely attained, Dong~\cite{d:16} proposed a regularized version of this semidefinite program, and demonstrated that it can be solved very efficiently through a specialized coordinate descent algorithm. Our work is partially inspired by the ideas proposed by Dong~\cite{d:16}. We refine his approach in several directions and make various theoretical and algorithmic contributions.

Our first contribution is a new class of convex QCP relaxations for~\eqref{problem_statement} constructed by using information from both $Q$ and the equality constraints $Ax=b$. These relaxations are derived from a semi-infinite program that generalizes the semi-infinite programming formulation proposed in~\cite{d:16}. Under our approach, we use the matrix $A^TA$ in order to modify the semidefinite constraint of the separation problem solved in~\cite{d:16}. This modification allows us to construct convex QCP relaxations which are at least as tight as those considered in~\cite{d:16}. Unlike the quadratic cuts introduced in~\cite{d:16}, the quadratic cuts obtained with our approach are nonconvex. However, as we show in \S\ref{sec:sicp_relaxations}, these nonconvex quadratic cuts define a convex feasible set when intersected with the equality constraints $Ax=b$, resulting in a convex relaxation of~\eqref{problem_statement}.

In our second contribution, we provide conditions under which our semi-infinite programming formulation is equivalent to a well-known SDP relaxation of~\eqref{problem_statement}. Moreover, we show that this SDP is the best relaxation in the class of SDP relaxations considered in this paper.

In our third contribution, we present a new analysis of the separation problem used to derive our quadratic cuts. In particular, we provide results on the finite attainment of this SDP by using its dual formulation. These results also apply to the separation problem in~\cite{d:16}, which is a special case of ours.

In our fourth contribution, we propose a new regularization approach for the semidefinite separation problem and modify the coordinate descent algorithm introduced in~\cite{d:16} accordingly. Through numerical experiments, we show that the quadratic cuts derived from our regularized separation problem provide a much better approximation of certain SDP bounds than the quadratic cuts obtained from the regularized separation problem proposed in~\cite{d:16}.

In order to assess the computational benefits of the proposed techniques, we implement the new quadratic relaxations in the global optimization solver BARON. These relaxations are incorporated into BARON's portfolio of relaxations and invoked according to a dynamic relaxation selection strategy introduced in~\cite{nrs:20}. We test our implementation on a large set of problems. Numerical results show that the new quadratic relaxations lead to a significant improvement in the performance of BARON, resulting in a new version of this solver which outperforms other state-of-the-art solvers such as CPLEX and GUROBI for many of our test problems.

The remainder of this paper is organized as follows. In \S\ref{convex_qcp_relaxations} we introduce the relaxations considered in this paper and investigate their theoretical properties. In \S\ref{separation_problem_analysis}, we provide a new analysis on the finite attainment of the semidefinite separation problem and present our regularization approach. In \S\ref{coordinate_minimization_algorithm} we introduce the version of coordinate minimization algorithm used to solve our regularized separation problem. This is followed by a description of our implementation in \S\ref{implementation}. In \S\ref{computational_results}, we present an extensive computational study which investigates the effectiveness of our regularization approach, the impact of the proposed relaxations on the performance of BARON, and the relative performance of several state-of-the-art global optimization solvers. Finally, in \S\ref{conclusions} we present conclusions from this work.

In what follows, we denote by $\R$ and $\R_{\geq 0}$, the set of real numbers and nonnegative real numbers, respectively. We use $\mathbbm{1} \in \R^n$ to denote a vector of ones. For a matrix $A \in \R^{m \times n}$, we use $A_{i \cdot}$, $A_{\cdot j}$ and $A_{ij}$ to denote its $i$-th row, $j$-th column and $(i,j)$-th entry, respectively. Let $M, N \in {\Sy}^{n}$, where $\Sy^{n}$ is the set of $n \times n$ real, symmetric matrices, and $N \succ 0$. We use $\lambda_{\text{min}} (M)$ to represent the smallest eigenvalue of $M$. Similarly, we denote by $\lambda_{\text{min}} (M, N)$ the smallest generalized eigenvalue of the problem $M v = \lambda N v$, where $v \in {\R}^{n}$. The inner product between matrices $M, P \in \Sy^{n}$ is denoted by $\langle M, P \rangle = \sum_{i=1}^n\sum_{j=1}^n M_{ij}P_{ij}$. Throughout this paper, we assume that $A$ has rank $m$ and use $Z \in \R^{n \times (n-m)}$ to denote an orthonormal basis for the nullspace of $A$.

\section{Construction and theoretical analysis of convex quadratic relaxations}
\label{convex_qcp_relaxations}

\subsection{A family of semi-infinite programming relaxations}
\label{sec:sicp_relaxations}

We start by considering the following reformulation of~\eqref{problem_statement}:
\begin{equation}
\label{reformulation_da_2}
\left.
\begin{array}{cl}
 \underset{x, y, v}{\text{min}}\;\; & v + q^T x \\
\st\;\;   & v \geq x^T \left( Q + \text{diag}(d) \right) x - d^T y + \alpha \|Ax - b\|^2 \\
		  & A x = b \\
		  & (x_i, y_i) \in \C_i, \; \forall i \in [n]       
\end{array}
\right\rbrace
\end{equation}
where $y \in \R^n$, $v \in \R$, $d \in \R^n$ is a vector used to perturb the diagonal entries of $Q$, $\alpha \in \R_{\geq 0}$, and $\C_i:= \{ (x_i, y_i) \in \R^2 : x_i \in S_i, \; y_i = x_i^2\}$. Define $L_i := \text{min} \{s \in \R : s \in S_i \}$ and $U_i := \text{max} \{s \in \R : s \in S_i \}$. It is simple to show that the convex hull of $\C_i$ is given by (see Proposition 1 in~\cite{d:16}):
\begin{equation}
\label{convex_hull_C}
\begin{array}{cl}
\conv (\C_i) = \{ (x_i, y_i) \in \R^2 : L_i \leq x_i \leq U_i, \; l_i(x_i) \leq y_i \leq u_i(x_i)\}          
\end{array}
\end{equation}
where $l_i(\cdot)$ is the tightest convex extension of $x_i^2$ when $x_i$ is restricted to $S_i$ (see \cite{ts02} for convex extensions) and $u_i(\cdot)$ is the concave envelope of $x_i^2$ over $[L_i, U_i]$. By replacing $\C_i$ with $\conv (\C_i)$ in~\eqref{reformulation_da_2}, we obtain the following relaxation of~\eqref{problem_statement}:
\begin{equation}
\left.
\label{relaxation_da}
\begin{array}{cl}
 \underset{(x, y) \in {\cal F}, v}{\text{min}}\;\; & v + q^T x  \\
\st\;\;   & v \geq x^T \left( Q + \text{diag}(d) \right) x - d^T y + \alpha \|Ax - b\|^2
\end{array}
\right\rbrace
\end{equation}
where ${\cal F} = \{x, y \in  {\R}^{n} : A x = b, \,\,\, (x_i, y_i) \in \conv (\C_i), \; \forall i \in [n] \}$. Let $\D_{\alpha} := \{d \in \R^n : Q + \text{diag}(d) + \alpha A^T A \succcurlyeq 0 \}$. Clearly, \eqref{relaxation_da} is a convex problem for any vector $d \in \D_{\alpha}$. By considering all such vectors, we obtain the following semi-infinite convex program (SICP):
\begin{subequations}
\label{sicp_da}
\begin{align}
 \underset{(x, y) \in {\cal F}, v}{\text{min}}\;\; & v + q^T x \label{sicp_da_obj} \\
\st\;\;   & v \geq x^T \left( Q + \text{diag}(d) \right) x - d^T y + \alpha \|Ax - b\|^2, \; \forall d \in \D_{\alpha}. \label{sicp_da_cuts}
\end{align}
\end{subequations}

Since any solution feasible in~\eqref{reformulation_da_2}, is feasible in~\eqref{sicp_da} as well, this SICP also is a relaxation of~\eqref{problem_statement}. To illustrate this, let $(\bar{x}, \bar{y}, \bar{v})$ be feasible to~\eqref{reformulation_da_2}. For each $i \in [n]$, we have $(\bar{x}_i, \bar{y}_i) \in \C_i \subseteq \conv (\C_i)$. Moreover, since $\bar{y}_i = \bar{x}_i^2, \; \forall i \in [n]$, and $A \bar{x} = b$, we have $\bar{v} = \bar{x}^T \left( Q + \text{diag}(d) \right) \bar{x} - d^T \bar{y} + \alpha \|A\bar{x} - b\|^2 = \bar{x}^T Q \bar{x}$. By replacing $(\bar{x}, \bar{y}, \bar{v})$ in~\eqref{sicp_da_cuts}, it is easy to see this inequality is satisfied $\forall d \in \D_{\alpha}$.

The quadratic term $\alpha \|Ax - b\|^2$ in~\eqref{sicp_da_cuts} vanishes for any $x$ feasible in~\eqref{sicp_da}. As a result, we can drop this term from~\eqref{sicp_da_cuts} in order to obtain the following simplified version of~\eqref{sicp_da}:
\begin{subequations}
\label{sicp_das}
\begin{align}
 \underset{(x, y) \in {\cal F}, v}{\text{{min}}}\;\; & v + q^T x \label{sicp_das_obj} \\
\st\;\;   & v \geq x^T \left( Q + \text{{diag}}(d) \right) x - d^T y, \; \forall d \in \D_{\alpha}. \label{sicp_das_cuts}
\end{align}
\end{subequations}
The quadratic cuts~\eqref{sicp_das_cuts} are not necessarily convex because, for $d \in \R^n$ and $\alpha \in \R_{\geq 0}$, the condition $Q + \text{diag}(d) + \alpha A^T A \succcurlyeq 0$ does not necessarily imply that $Q + \text{diag}(d) \succcurlyeq 0$. However, the quadratic cuts~\eqref{sicp_das_cuts} need not be convex in order for~\eqref{sicp_das} to be a convex optimization problem. As we show in the next proposition, the quadratic cuts~\eqref{sicp_das_cuts} define a convex feasible set when intersected with the constraints $Ax = b$, making~\eqref{sicp_das} a convex relaxation of~\eqref{problem_statement}.
\begin{proposition}
\label{sicp_da_convexity}
The semi-infinite program~\eqref{sicp_das} is a convex optimization problem.
\end{proposition}
\begin{proof}
This proof relies on the projection of the feasible set of~\eqref{sicp_das} onto the nullspace of $A$. Let $\H = \{x \in \R^n : Ax = b \}$. Clearly, any point satisfying $Ax = b$ can be expressed as $x = \hat{x} + Z x_z$, where $\hat{x} \in \H$ and $x_z \in \R^{n-m}$. By using this transformation, \eqref{sicp_das} can be equivalently written as:
\begin{subequations}
\label{sicp_das_proj}
\begin{align}
 \underset{x_z, y, v}{\text{min}}\;\; & v + q^T \left(\hat{x} + Z x_z\right) \label{sicp_das_proj_obj} \\
\st\;\;   & v \geq {\left(\hat{x} + Z x_z\right)}^T \left( Q + \text{diag}(d) \right) \left(\hat{x} + Z x_z\right) - d^T y, \; \forall d \in \D_{\alpha} \label{sicp_das_proj_c1} \\
		  & L_i \leq \hat{x}_i + e_i^T Z x_z \leq U_i, \; \forall i \in [n] \label{sicp_das_proj_c2} \\
		  & l_i\left(\hat{x}_i + e_i^T Z x_z\right) \leq y_i \leq u_i\left(\hat{x}_i + e_i^TZx_z\right), \; \forall i \in [n]. \label{sicp_das_proj_c3}
\end{align}
\end{subequations}
To prove that~\eqref{sicp_das_proj} is convex, it suffices to show that $Z^T \left(Q + \text{diag}(d)\right) Z$ is positive semidefinite. By definition, any vector $d \in \D_{\alpha}$ satisfies:
\begin{equation}
\label{semidefinite_condition_1}
\begin{array}{cl}
w^T \left(Q + \text{diag}(d) + \alpha A^T A \right) w \geq 0, \; \forall w \in \R^n .
\end{array}
\end{equation}
Let $w = Z w_z$, where $w_z \in \R^{n-m}$. For this choice of $w$, \eqref{semidefinite_condition_1} becomes
\begin{equation}
\label{semidefinite_condition_2}
\begin{array}{cl}
w_z^T Z^T \left(Q + \text{diag}(d) \right) Z w_z \geq 0, \; w_z \in \R^{n-m} .
\end{array}
\end{equation}
Clearly, \eqref{semidefinite_condition_2} holds for all vectors $w_z \in \R^{n-m}$. Hence, $Z^T \left(Q + \text{diag}(d)\right) Z$ is positive semidefinite for any $d \in \D_{\alpha}$. This completes the proof.
\end{proof}

By setting $\alpha = 0$ in~\eqref{sicp_das}, we obtain the following SICP relaxation of~\eqref{problem_statement} which was considered in~\cite{d:16}:
\begin{subequations}
\label{sicp_d}
\begin{align}
 \underset{(x, y) \in {\cal F}, v}{\text{min}}\;\; & v + q^T x \label{sicp_d_obj} \\
\st\;\;   & v \geq x^T \left( Q + \text{diag}(d) \right) x - d^T y, \; \forall d \in \D_{0} \label{sicp_d_cuts}
\end{align}
\end{subequations}
where $\D_{0} := \{d \in \R^n : Q + \text{diag}(d) \succcurlyeq 0 \}$. The formulation in~\eqref{sicp_d} served as a motivation in~\cite{d:16} to develop an algorithm to construct convex relaxations for~\eqref{problem_statement} by using a finite number of quadratic cuts of the form~\eqref{sicp_d_cuts}. As we demonstrate in \S\ref{sec:cutting_surface_algorithm}, a similar algorithm can be devised based on the SICP~\eqref{sicp_das}.

In~\eqref{sicp_das}, the set $\D_{\alpha}$ is parameterized by the scalar $\alpha$. An interesting question that arises in this context is how we can choose $\alpha$ to obtain the tightest relaxation of the form~\eqref{sicp_das}. This question is addressed by the following proposition.
\begin{proposition}
\label{sicp_da_tightest_relaxation}

Let $\alpha_1$ and $\alpha_2$ be real scalars such that $0 \leq \alpha_1 \leq \alpha_2$. Denote by $\mu_{\emph{\text{SICPda1}}}$ and $\mu_{\emph{\text{SICPda2}}}$ the optimal objective function values in the SICP~\eqref{sicp_das} for $\alpha_1$ and $\alpha_2$, respectively. Define $\D_{\infty} := \{ d \in \R^n : Z^T(Q + \textnormal{\text{diag}}(d))Z \succcurlyeq 0 \}$. Then, the following holds:

\begin{enumerate}[(i)]
\item $\mu_{\emph{\text{SICPda2}}} \geq \mu_{\emph{\text{SICPda1}}}$.
\item The tightest relaxation of the form~\eqref{sicp_das} is obtained when $\alpha \to\infty$.
\item $\lim_{\alpha \rightarrow \infty} \D_{\alpha} = \D_{\infty}$.
\end{enumerate}

\end{proposition}
\begin{proof}

We start with the proof of (i). Let $\D_{\alpha_1}$ and $\D_{\alpha_2}$ be the sets of diagonal perturbations parameterized by $\alpha_1$ and $\alpha_2$, respectively. To prove the claim in (i), it suffices to show that $\D_{\alpha_1} \subseteq \D_{\alpha_2}$. Let $\bar{d} \in \D_{\alpha_1}$. By definition, $\bar{d}$ satisfies:
\begin{equation}
\label{semidefinite_condition_3}
\begin{array}{cl}
w^T \left(Q + \text{diag}(\bar{d}) + \alpha_1 A^T A \right) w \geq 0, \; \forall w \in \R^n .
\end{array}
\end{equation}
As $A^T A \succcurlyeq 0 $ and $\alpha_2 - \alpha_1 \geq 0$, we have that $(\alpha_2 - \alpha_1) w^T A^T A w \geq 0, \; \forall w \in \R^n$. This condition combined with~\eqref{semidefinite_condition_3} implies
\begin{equation}
\label{semidefinite_condition_4}
\begin{array}{cl}
w^T \left(Q + \text{diag}(\bar{d}) + \alpha_2 A^T A \right) w \geq 0, \; \forall w \in \R^n .
\end{array}
\end{equation}
It follows that $\bar{d} \in \D_{\alpha_2}$. Hence, $\D_{\alpha_1} \subseteq \D_{\alpha_2}$, which completes the proof of (i). The claim in (ii) follows directly from (i). To prove (iii), it suffices to show that, for any $\bar{d} \in \D_{\infty}$, the following condition holds:
\begin{equation}
\label{define_limit_alpha2infty}
    \lim\limits_{\alpha \rightarrow \infty} 
    w^T \left(Q + \text{diag}(\bar{d}) + \alpha A^TA \right) w \geq 0, \; \forall w \in \R^n .
\end{equation}
Clearly, any $w \in \R^n$ can be written as $w = w_A + Z w_z$, where $w_A \in \text{range}(A^T)$ and $w_z \in \R^{n-m}$. Suppose that $w_A \neq 0$. Then, $w^T A^T A w = w_A^T A^T A w_A > 0$, and it is easy to show that~\eqref{define_limit_alpha2infty} holds in the limit as $\alpha \rightarrow \infty$. Now, assume that $w_A = 0$. Since $AZ = 0$, the left-hand side of~\eqref{define_limit_alpha2infty} reduces to $w_z^T Z^T\left( Q + \text{diag}(\bar{d}) \right) Z w_z$, which is nonnegative because $\bar{d} \in D_{\infty}$. This proves the claim in (iii).
\end{proof}

A direct consequence of Proposition~\ref{sicp_da_tightest_relaxation}(i) is that, for any $\alpha > 0$, the bound provided by~\eqref{sicp_das} is at least as large as that given by~\eqref{sicp_d}.

\subsection{Relationship between the semi-infinite and semi-definite formulations}
\label{sec:sicp_sdp_relationship}

In~\cite{d:16}, it was shown that the SICP~\eqref{sicp_d} is equivalent to the following SDP relaxation of~\eqref{problem_statement} (see Section 2 in~\cite{d:16} for details):
\begin{subequations}
\label{sdp_d}
\begin{align}
\underset{(x, y) \in {\cal F}, X}{\text{min}}\;\; & \langle Q, X \rangle + q^T x \label{sdp_d_obj} \\
\st\;\;     & X - x x^T \succcurlyeq 0 \label{sdp_d_c1} \\
 			& X_{ii} = y_i, \; \forall i \in [n]. \label{sdp_d_c2}
\end{align}
\end{subequations}
Motivated by this result, in this section we investigate the relationship between the SICP~\eqref{sicp_das} and the following SDP relaxation of~\eqref{problem_statement}:
\begin{subequations}
\label{sdp_da}
\begin{align}
\underset{(x,y) \in {\cal F}, X}{\text{min}}\;\; & \langle Q, X \rangle + q^T x \label{sdp_da_obj} \\
\st\;\;     & X - x x^T \succcurlyeq 0 \label{sdp_da_c1} \\
 			& X_{ii} = y_i, \; \forall i \in [n] \label{sdp_da_c2} \\
 			& \langle A^T A, X \rangle - 2 (A^T b)^T x + b^T b  = 0. \label{sdp_da_c3}
\end{align}
\end{subequations}

We start by showing that, for any $\alpha \geq 0$, the optimal solution of the SICP~\eqref{sicp_das} is bounded by the optimal solutions of the SDPs~\eqref{sdp_d} and~\eqref{sdp_da}.

\begin{proposition}
\label{relationship_sdp_da_sicp_da}
Assume that $\alpha \geq 0$ in~\eqref{sicp_das}. Denote by $\mu_{\emph{\text{SICPda}}}$, $\mu_{\emph{\text{SDPd}}}$ and $\mu_{\emph{\text{SDPda}}}$ the optimal objective function values in~\eqref{sicp_das}, \eqref{sdp_d} and~\eqref{sdp_da}, respectively. Then, $\mu_{\emph{\text{SDPd}}} \leq \mu_{\emph{\text{SICPda}}} \leq \mu_{\emph{\text{SDPda}}}$.
\end{proposition}

\begin{proof}

We start by proving that $\mu_{{\text{SDPd}}} \leq \mu_{{\text{SICPda}}}$. Let $\mu_{{\text{SICPd}}}$ be the optimal solution of the SICP~\eqref{sicp_d}. Since $\alpha \geq 0$, Proposition~\ref{sicp_da_tightest_relaxation}(i) implies that $\mu_{{\text{SICPd}}} \leq \mu_{{\text{SICPda}}}$. By Theorem 1 in~\cite{d:16} we have that $\mu_{{\text{SICPd}}} = \mu_{{\text{SDPd}}}$. Hence, $\mu_{{\text{SDPd}}} \leq \mu_{{\text{SICPda}}}$. To prove that $\mu_{{\text{SICPda}}} \leq \mu_{{\text{SDPda}}}$, it suffices to show that for any $(\bar{x}, \bar{y}, \bar{X})$ feasible in~\eqref{sdp_da} and $\bar{d} \in \D_{\alpha}$, the following condition holds:
\begin{equation}
\label{relationship_sdp_da_sicp_da_1}
\langle Q, \bar{X} \rangle \geq \bar{x}^T (Q + \text{diag}(\bar{d})) \bar{x} - \bar{d}^T \bar{y}.
\end{equation}
To this end, consider the following inequality:
\begin{equation}
\label{relationship_sdp_da_sicp_da_2}
\langle Q + \text{diag}(\bar{d}) + \alpha A^T A , \bar{X} - \bar{x}\bar{x}^T \rangle \geq 0
\end{equation}
which is valid by the feasibility of $\bar{x}$ and $\bar{X}$ in~\eqref{sdp_da} and the self-duality of the positive semi-definite cone. This inequality can be equivalently written as:
\begin{equation}
\label{relationship_sdp_da_sicp_da_3}
\langle Q, \bar{X} \rangle \geq \bar{x}^T (Q + \text{diag}(\bar{d})) \bar{x} - \sum_{i = 1}^{n} d_i \bar{X}_{ii} - \alpha \langle A^T A, \bar{X} \rangle + \alpha \bar{x}^T A^T A \bar{x}.
\end{equation}
From the feasibility of $(\bar{x}, \bar{y}, \bar{X})$ in~\eqref{sdp_da}, it follows that $\bar{X}_{ii} = \bar{y}_i$, $\forall i \in [n]$, and $\langle A^T A, \bar{X} \rangle = 2 b^T A \bar{x} - b^T b$. Then, the inequality in~\eqref{relationship_sdp_da_sicp_da_3} becomes:
\begin{equation}
\label{relationship_sdp_da_sicp_da_4}
\langle Q, \bar{X} \rangle \geq \bar{x}^T (Q + \text{diag}(\bar{d})) \bar{x} - \bar{d}^T \bar{y} + \alpha {\left( A\bar{x} - b \right)}^T \left( A\bar{x} - b \right).
\end{equation}
Since $A\bar{x} - b = 0$, \eqref{relationship_sdp_da_sicp_da_4} is equivalent to~\eqref{relationship_sdp_da_sicp_da_1}, which completes the proof.
\end{proof}

Now, we consider the case in which $\alpha \rightarrow \infty$ in the SICP~\eqref{sicp_das}. By Proposition~\ref{sicp_da_tightest_relaxation}(iii), the resulting SICP can be written as:
\begin{subequations}
\label{sicp_da_inf}
\begin{align}
 \underset{(x, y) \in {\cal F}, v}{\text{{min}}}\;\; & v + q^T x \label{sicp_da_inf_obj} \\
\st\;\;   & v \geq x^T \left( Q + \text{{diag}}(d) \right) x - d^T y, \; \forall d \in \D_{\infty}. \label{sicp_da_inf_cuts}
\end{align}
\end{subequations}
Proposition~\ref{sicp_da_tightest_relaxation}(ii) implies that~\eqref{sicp_da_inf} is the tightest relaxation of the form~\eqref{sicp_das} that can be constructed for~\eqref{problem_statement}. Moreover, from Proposition~\ref{relationship_sdp_da_sicp_da}, we know that~\eqref{sdp_da} provides an upper bound on the optimal solution of~\eqref{sicp_da_inf}. Therefore, an important question is the existence of conditions under which these two relaxations are equivalent. This question is addressed in the remainder of this section.

In~\cite{d:16}, the equivalence between the SICP~\eqref{sicp_d} and the SDP~\eqref{sdp_d} was established by applying strong duality to the SDP~\eqref{sdp_d}. However, unlike the SDP~\eqref{sdp_d}, the SDP~\eqref{sdp_da} does not admit a strictly feasible solution. To illustrate this, note that, for any $x$ satisfying $Ax=b$, \eqref{sdp_da_c3} can be equivalently written as:
\begin{subequations}
\label{sdp_da_eq_soln}
\begin{align}
 & \langle A^T A, X \rangle - {\left(2 A^T b \right)}^T x + b^T b + \langle A^T A, xx^T \rangle - \langle A^T A, xx^T \rangle = 0 \label{sdp_da_eq_soln_c1} \\
\implies & \langle A^T A, X - xx^T \rangle = 0 \label{sdp_da_eq_soln_c2}
\end{align}
\end{subequations}
which implies that $X - xx^T$ cannot be positive definite for the pairs $(x,X)$ that are feasible in~\eqref{sdp_da}. It follows that we cannot apply strong duality to~\eqref{sdp_da}. As a result, in order to show the equivalence between the SICP~\eqref{sicp_da_inf} and the SDP~\eqref{sdp_da}, we will rely on an auxiliary SDP which we will derive from~\eqref{sdp_da}.

We begin by characterizing the set of symmetric matrices $X \in \Sy^{n}$ which are feasible in~\eqref{sdp_da}.

\begin{proposition}
\label{X_characterization}
Let the $\bar{x} \in \R^{n}$ and $\bar{X} \in \Sy^{n}$ be feasible in the {SDP}~\eqref{sdp_da}. Then $\bar{X}$ has the form $\bar{X} = \bar{x}\bar{x}^T + Z \bar{W}_z Z^T$, where $\bar{W}_z \in \Sy^{n-m}$ and $\bar{W}_z \succeq 0$.
\end{proposition}

\begin{proof}

For any $(\bar{x},\bar{X})$ feasible in~\eqref{sdp_da} we can express $\bar{X}$ as $\bar{X} = \bar{x} \bar{x}^T + \widetilde{X}$ with $\widetilde{X} \in \Sy^{n}$ and $\widetilde{X} \succeq 0$. As mentioned previously, the constraint \eqref{sdp_da_c3} implies that $\langle A^T A, \bar{X} - \bar{x} \bar{x}^T \rangle = 0$, which can be equivalently written as $\langle A^TA, \widetilde{X} \rangle = 0$. Let the eigenvalue decomposition of $\widetilde{X}$ be given by $\widetilde{X} = \sum_{i=1}^p \lambda_i v_iv_i^T$, where $p \leq n$. Since $\widetilde{X} \succeq 0$, we have that $\lambda_i > 0 ,\; \forall i =1, \ldots, p$. By using this decomposition in $\langle A^TA, \widetilde{X} \rangle = 0$ we obtain 
\begin{subequations}
\label{X_characterization_condition}
\begin{align}
 & \langle A^TA , \widetilde{X} \rangle = \sum_{i=1}^p \lambda_i ||Av_i\|^2 = 0 \label{X_characterization_condition_c1} \\
\implies & Av_i = 0, \; \forall i =1, \ldots, p. \label{X_characterization_condition_c2}
\end{align}
\end{subequations}
This shows that $\widetilde{X}$ is of the form $Z \bar{W_z} Z^T$ with $\bar{W}_z \in \Sy^{n-m}$ and $\bar{W_z} \succeq 0$.
\end{proof}

Now, let $\H = \{x \in \R^n : Ax = b \}$. Clearly, any point satisfying $Ax = b$ can be expressed as 
\begin{equation}
\label{x_sdp_da}
\begin{aligned}
x = \hat{x} + Z x_z
\end{aligned}
\end{equation}
where $\hat{x} \in \H$ and $x_z \in \R^{n-m}$. From Proposition~\ref{X_characterization} we know that any $X$ feasible in~\eqref{sdp_da} is of the form ${X} = {x}{x}^T + Z {W}_z Z^T$, where ${W}_z \in \Sy^{n-m}$ and ${W}_z \succeq 0$. By using~\eqref{x_sdp_da} in the expression for $X$ we obtain $X = \hat{x}\hat{x}^T + \hat{x}(Z x_z)^T + (Z x_z)\hat{x}^T + Z x_z x_z^T Z^T + Z W_z Z^T$. Define $X_z := W_z + x_z x_z^T$ with $X_z \in \Sy^{n-m}$. Then $X$ can be expressed as
\begin{equation}
\label{X_sdp_da}
\begin{aligned}
X = \hat{x}\hat{x}^T + \hat{x}(Z x_z)^T + (Z x_z)\hat{x}^T + Z X_z Z^T.
\end{aligned}
\end{equation}
By substituting~\eqref{x_sdp_da} and~\eqref{X_sdp_da} in~\eqref{sdp_da}, we can cast this SDP as follows:
\begin{subequations}
\label{sdp_daz}
\begin{align}
\underset{x_z, y, X_z}{\text{min}}\;\; & \left\langle Q, Z X_z Z^T \right\rangle + \left(2Q\hat{x} + q\right)^T (Zx_z) + \hat{x}^TQ\hat{x} + q^T\hat{x}\label{sdp_da_z_obj} \\
\st\;\;     & X_z - x_z x_z^T \succcurlyeq 0 \label{sdp_daz_c1} \\
 			& \hat{x}_i^2 + 2\hat{x}_i \left( e_i^TZx_z \right) + e_i^TZX_zZ^Te_i = y_i, \; \forall i \in [n] \label{sdp_daz_c2} \\
 			& L_i \leq \hat{x}_i + e_i^T Z x_z \leq U_i, \; \forall i \in [n] \label{sdp_daz_c3} \\
 			& l_i\left(\hat{x}_i + e_i^T Z x_z\right) \leq y_i \leq u_i\left(\hat{x}_i + e_i^TZx_z\right), \; \forall i \in [n]. \label{sdp_daz_c4}
\end{align}
\end{subequations}
As we show later in this section, under certain conditions, the SDP in~\eqref{sdp_daz} admits a strictly feasible solution. Based on this key observation, in the next theorem we rely on strong duality holding for~\eqref{sdp_daz} in order to show the equivalence between the SICP~\eqref{sicp_da_inf} and the SDP~\eqref{sdp_da}.

\begin{theorem}
\label{sdp_sicp_equivalence}
Let $\mu_{\textnormal{SDPda}}$ and $\mu_{\textnormal{SICPda}\infty}$ denote the optimal objective function values in~\eqref{sdp_da} and~\eqref{sicp_da_inf}, respectively. Assume that the {SDP}~\eqref{sdp_daz} admits a strictly feasible solution. Then, $\mu_{\textnormal{SDPda}} = \mu_{\textnormal{SICPda}\infty}$.
\end{theorem}

\begin{proof}
To prove that $\mu_{\textnormal{SDPda}} = \mu_{\textnormal{SICPda}\infty}$, we rely on the SDP~\eqref{sdp_daz} and its dual. Denote by $\mu_{\textnormal{SDPdaz}}$ and $\mu_{\textnormal{DSDPdaz}}$, the optimal objective function values of the SDP~\eqref{sdp_daz} and its dual, respectively. Since~\eqref{sdp_daz} admits a strictly feasible solution, Slater’s condition is satisfied by~\eqref{sdp_daz}, which implies that strong duality holds for this SDP and its dual. Hence, $\mu_{\textnormal{SDPdaz}} = \mu_{\textnormal{DSDPdaz}}$. Moreover, as the SDPs~\eqref{sdp_da} and~\eqref{sdp_daz} are equivalent, we have that $\mu_{\textnormal{SDPda}} = \mu_{\textnormal{SDPdaz}} = \mu_{\textnormal{DSDPdaz}}$.

Now, we construct the dual of~\eqref{sdp_daz}. Let $d_i \in \R, \; i \in [n]$, be the multipliers associated with the constraints~\eqref{sdp_daz_c2}. Then, the dual of~\eqref{sdp_daz} is given by
\begin{equation}
\label{dsdp_daz_1}
\underset{d \in \R^n}{\text{max}} \left\lbrace
\begin{aligned}
\centering
& \underset{x_z, y, X_z}{\text{min}}
& & \langle Q_{d,Z}, X_z \rangle + q_{d, \hat{x}}^T (Zx_z) + k_{d, \hat{x}} - d^Ty \\
& \;\;\; \text{s.t.} & & X_z - x_z x_z^T \succcurlyeq 0 \\
& & & L_i \leq \hat{x}_i + e_i^T Z x_z \leq U_i, \; \forall i \in [n] \\
& & & l_i\left(\hat{x}_i + e_i^T Z x_z\right) \leq y_i \leq u_i\left(\hat{x}_i + e_i^TZx_z\right), \; \forall i \in [n]
\end{aligned}
\right\rbrace
\end{equation}
where $Q_{d,Z} = Z^T Q_d Z$, $Q_d = Q+\text{diag}(d)$, $q_{d, \hat{x}} = 2 Q_d \hat{x} + q$, and $k_{d, \hat{x}} = \hat{x}^T Q_d \hat{x} + q^T\hat{x}$. For the minimization problem in~\eqref{dsdp_daz_1} to be bounded below, we need to choose $d$ such that $Q_{d,Z} \succcurlyeq 0$, i.e. $d \in \D_{\infty} := \{ d \in \R^n : Z^T(Q + \textnormal{\text{diag}}(d))Z \succcurlyeq 0 \}$. This restriction on $d$ implies that $X_z = x_zx_z^T$ holds at any optimal solution to the inner minimization problem. As a result, the dual problem in~\eqref{dsdp_daz_1} can be simplified as:
\begin{equation}
\label{dsdp_daz_2}
\underset{d \in \D_{\infty}}{\text{max}} \left\lbrace
\begin{aligned}
\underset{x_z, y}{\text{min}}\;\; & 
\left(\hat{x} + Zx_z\right)^T Q_d \left(\hat{x} + Z z_x\right) + 
q^T (\hat{x} + Zx_z) - d^Ty \\
\st\;\;     & L_i \leq \hat{x}_i + e_i^T Z x_z \leq U_i, \; \forall i \in [n] \\
 			& l_i\left(\hat{x}_i + e_i^T Z x_z\right) \leq y_i \leq u_i\left(\hat{x}_i + e_i^TZx_z\right), \; \forall i \in [n].
\end{aligned}
\right\rbrace
\end{equation}
Since strong duality holds for~\eqref{sdp_daz} and its dual~\eqref{dsdp_daz_2}, both problems attain their optimal objective functions values. This implies that $\exists \, d^* \in \D_{\infty}$ such that:
\begin{equation}
\label{dsdp_daz_3}
\left.
\begin{aligned}
\mu_\textnormal{{DSDPdaz}} = \underset{x_z, y, v}{\text{min}}\;\; & v + q^T \left(\hat{x} + Zx_z\right) \\
\st\;\;     & v \geq \left(\hat{x} + Zx_z\right)^T \left(Q+\text{diag}(d^*)\right) \left(\hat{x} + Z x_z\right) - {d^*}^Ty \\
            & L_i \leq \hat{x}_i + e_i^T Z x_z \leq U_i, \; \forall i \in [n] \\
 			& l_i\left(\hat{x}_i + e_i^T Z x_z\right) \leq y_i \leq u_i\left(\hat{x}_i + e_i^TZx_z\right), \; \forall i \in [n].
\end{aligned}
\right\rbrace
\end{equation}
It is easy to show that~\eqref{dsdp_daz_3} is a relaxation of~\eqref{sicp_da_inf}. To this end, note that by replacing $x$ and $\D_{\infty}$ in~\eqref{sicp_da_inf} with $\hat{x} + Zx_z$ and $\{d^*\}$, respectively, we obtain~\eqref{dsdp_daz_3}. Hence, $\mu_{\textnormal{SICPda}\infty} \geq \mu_{\textnormal{DSDPdaz}} = \mu_{\textnormal{SDPdaz}} = \mu_{\textnormal{SDPda}}$. Combining this condition with the result of Proposition~\ref{relationship_sdp_da_sicp_da}, we obtain $\mu_{\textnormal{SICPda}\infty} = \mu_{\textnormal{DSDPda}}$, which completes the proof.
\end{proof}
The key assumption in Theorem~\ref{sdp_sicp_equivalence} is the existence of a strictly feasible solution for the SDP~\eqref{sdp_daz}. In the next proposition, we provide a set of sufficient conditions under which a strictly feasible solution can be constructed for~\eqref{sdp_daz}.
\begin{proposition}
\label{sufficient_conditions}
Define ${\cal X} := \{x \in  {\R}^{n} : A x = b, \; L_i < x_i < U_i, \; \forall i \in [n] \}$. Assume that ${\cal X}$ is nonempty and that the following conditions hold:
\begin{enumerate}[(i)]
    \item $e_i^TZ \neq 0, \; \forall i \in [n]$.
    \item $\exists \, \widetilde{x} \in {\cal X}$ such that $l_i(\widetilde{x}_i) = \widetilde{x}_i^2$ and $l_i(\widetilde{x}_i) < u_i(\widetilde{x}_i), \; \forall i \in [n]$.
\end{enumerate}
Then, the SDP~\eqref{sdp_daz} admits a strictly feasible solution.
\end{proposition}
\begin{proof}
Define the scalar $\delta$ as: 
\begin{equation}
    \delta := \min\limits_{i \in [n]} \frac{u_i(\tilde{x}_i) - l_i(\tilde{x}_i)}{e_i^T Z Z^T e_i}.
    \label{define_delta}
\end{equation}
From assumption (i), it follows that $e_i^T Z Z^T e_i > 0, \; \forall i \in [n]$. Moreover, assumption (ii) implies that $u_i(\tilde{x}_i) - l_i(\tilde{x}_i) > 0, \; \forall i \in [n]$. Hence, $0 < \delta < \infty$. Let $\hat{x} = \tilde{x}$, $\bar{x}_z = 0$, $\bar{y}_i = \tilde{x}_i^2 + \epsilon e_i^TZZ^Te_i, \; i \in [n]$, and $\bar{X}_z = \epsilon I_{n-m}$, where $\tilde{x}$ satisfies the condition in (ii), and $0 < \epsilon < \delta$. It is simple to check that this choice is feasible in the SDP~\eqref{sdp_daz}, and further, the inequalities~\eqref{sdp_daz_c1}, \eqref{sdp_daz_c3} and~\eqref{sdp_daz_c4} are satisfied strictly. This completes the proof.
\end{proof}

If for a given index $i$ assumption (i) in Proposition~\ref{sufficient_conditions} fails to hold, then $x_i$ must be a fixed variable in the original problem. To illustrate this, note that any $x \in \R^n$ satisfying $Ax=b$ can be written as $x = \hat{x} + Z x_z$, where $\hat{x}$ satisfies $A \hat{x} = b$ and $x_z \in \R^{n-m}$. The condition $e_i^T Z = 0$ means that $e_i^T Z x_z = 0$, which in turn implies that $x_i = \hat{x}_i$. Hence, for all the indices $i$ for which $e_i^T Z = 0$, the corresponding variables can be eliminated from the original problem in order to obtain a reduced problem for which assumption (i) holds.

The satisfaction of assumption (ii) in Proposition~\ref{sufficient_conditions} depends on the form of the functions $l_i(x_i)$ and $u_i(x_i)$. It is easy to show that $u_i(x_i)$ is an affine function given by $u_i(x) = (L_i+U_i)x_i - L_i U_i$. On the other hand, the form of $l_i(x_i)$ depends on the structure of the set $S_i$. If $S_i$ is given by a closed interval, then $l_i(x_i)= x_i^2$, and in this case assumption (ii) is satisfied provided that there exists a vector $x \in {\R}^{n}$ such that $Ax = b$ and $L_i < x_i < U_i, \; \forall i \in [n]$. If $S_i$ is given by a set of two discrete points, then $l_i(x_i) = u_i(x_i), \; \forall x_i \in [L_i, U_i]$, and assumption (ii) fails to hold. This occurs, for example, when $x_i$ is binary, i.e., $S_i = \{0, 1\}$, because in this case $l_i(x_i) = u_i(x_i) = x_i, \; \forall x_i \in [0,1]$.

\subsection{Further insights into the semidefinite relaxation}

The SDP relaxation~\eqref{sdp_da} can be derived from~\eqref{sdp_d} by adding the valid equality $(Ax-b)^T (Ax-b) = 0$ and lifting it into the space of $(x,X)$. Clearly, we can construct other SDP relaxations for~\eqref{problem_statement} by including other classes of constraints derived from $Ax = b$. In general, we can apply the following procedure:

\begin{enumerate}[(R1)]
    \item\label{recipe:step1} identify a (possibly empty) set $\J$ of quadratic functions of the form $f_j(x) = x^T C_j x + c_j^T x + \gamma_j$, where $C_j \in \Sy^n, c_j \in \R^n, \gamma_j \in \R$, such that $f_j(x) = 0$ for $x \in \Omega := \{x \in \R^n \,|\, Ax = b\}$;
    \item\label{recipe:step2} construct an SDP relaxation for~\eqref{problem_statement} as
\begin{subequations}
\label{sdp_daJ}
\begin{align}
\underset{(x, y) \in {\cal F}, X}{\text{min}}\;\; & \langle Q, X \rangle + q^T x \label{sdp_daJ_obj} \\
\st\;\;     & X - x x^T \succcurlyeq 0 \label{sdp_daJ_c1} \\
 			& X_{ii} = y_i, \; \forall i \in [n] \label{sdp_daJ_c2} \\
 			& \langle C_j, X \rangle + c_j^T x + \gamma_j  = 0, \; \forall j \in \J \label{sdp_daJ_c4}
\end{align}
\end{subequations}
where the constraints~\eqref{sdp_daJ_c4} are obtained by lifting the valid equalities $x^T C_j x + c_j^T x + \gamma_j = 0, \; \forall j \in \J$ into the space of $(x,X)$.
\end{enumerate}

There are different types of functions $f_j(x)$ that satisfy the condition in~(R\ref{recipe:step1}). Some examples that have been considered in the literature are~\cite{fr:07}: $\left(x_j(A_{i\cdot}x - b_i)\right)$, and $\left((A_{j\cdot}x - b_j)(A_{i\cdot}x-b_i)\right)$, $\left(x^TA_{j\cdot}^TA_{i\cdot}x - b_jb_i\right)$. A natural question in this context is whether we can improve on the bound given by~\eqref{sdp_da} when restricted to the class of relaxations in~\eqref{sdp_daJ}. We address this question in the remainder of this section by showing that~\eqref{sdp_da} is the best relaxation among the class of relaxations in~\eqref{sdp_daJ}.

We start by recalling the properties of the functions satisfying~(R\ref{recipe:step1}). The following result follows from Theorem 1 in~\cite{fr:07}.
\begin{proposition}\label{prop:null_quad}
Let $f(x) = x^T C x + c^T x + \gamma$ be a quadratic function. Then, $f(x) = 0$ for all $x \in \Omega := \{ x \in \R^n \,|\, Ax = b\}$ if and only if $C = A^TW^T + WA$, $c = A^T \beta - 2Wb$, and $\gamma = -b^T \beta$ for some $W \in \R^{n \times m}$ and $\beta \in \R^m$.
\end{proposition}
Based on Proposition~\ref{prop:null_quad}, we can assume without loss of generality that $C_j = A^TW_j^T + W_jA$ for some $W_j \in \R^{n \times m}$. In the next proposition, we show that the feasible set of~\eqref{sdp_da} is a subset of the feasible set of~\eqref{sdp_daJ}. Moreover, we provide conditions on the choice of quadratic functions in $\J$ for these two feasible sets to be identical. 

\begin{proposition}\label{prop:sdp_relation_Z_J}
Let ${\cal F}_{\emph{\text{SDPda}}}$ and ${\cal F}_{\emph{\text{SDPdaJ}}}$ denote the feasible regions of the SDPs in~\eqref{sdp_da} and~\eqref{sdp_daJ}, respectively. Then, the following holds:
\begin{enumerate}[(i)]
    \item ${\cal F}_{\emph{\text{SDPda}}} \subseteq {\cal F}_{\emph{\text{SDPdaJ}}}$.
    \item If $\exists \, \omega_j, j \in \J$ such that $\sum_{j \in \J} \omega_j W_j = A^T$ then ${\cal F}_{\emph{\text{SDPda}}} = {\cal F}_{\emph{\text{SDPdaJ}}}$.
\end{enumerate}
\end{proposition}
\begin{proof}
We first prove (i). From Proposition~\ref{X_characterization}, any $(\bar{x},\bar{X}) \in {\cal F}_{\text{SDPda}}$ satisfies $\bar{X} = \bar{x} \bar{x}^T + Z \bar{W}_z Z^T$, where $\bar{W}_z \in \Sy^{n-m}$ and $\bar{W}_z \succeq 0$. Then, for any $(\bar{x}, \bar{y}, \bar{X}) \in {\cal F}_{\text{SDPda}}$ and $\forall j \in \J$ we have:
\begin{subequations}
\label{sdp_da_daJ_1}
\begin{align}
 \;\; & \langle C_j, \bar{X} \rangle + c_j^T \bar{x} + \gamma_j  \label{sdp_da_daJ_1_c1} \\
=\;\; & \langle C_j, \bar{X} - \bar{x}\bar{x}^T \rangle + \bar{x}^T C_j \bar{x} + c_j^T \bar{x} + \gamma_j  \label{sdp_da_daJ_1_c2} \\
=\;\; & \langle C_j, \bar{X} - \bar{x}\bar{x}^T \rangle = \langle A^T W_j^T + W_j A, Z \bar{W}_z Z^T \rangle \label{sdp_da_daJ_1_c3} = 0
\end{align}
\end{subequations}
where~\eqref{sdp_da_daJ_1_c2} follows from adding 
and subtracting $\bar{x}^T C_j \bar{x}$, the first equality in~\eqref{sdp_da_daJ_1_c3} follows from~(R\ref{recipe:step1}), the second equality in~\eqref{sdp_da_daJ_1_c3} from Proposition~\ref{prop:null_quad} and the final equality from the fact that $Z$ is a basis for the nullspace of $A$. Thus $(\bar{x},\bar{y},\bar{X}) \in {\cal F}_{\text{SDPdaJ}}$ proving the claim in (i).

Now, we prove the claim in (ii). Assume that there exist $\omega_j, j \in \J$ such that the condition in (ii) holds. By performing a linear combination of the inequalities in~\eqref{sdp_daJ_c4} using $\omega_j$, we obtain that for any $(\bar{x},\bar{y},\bar{X}) \in {\cal F}_{\text{SDPdaJ}}$: 
\begin{subequations}
\label{sdp_da_daJ_2}
\begin{align}
0 = \;\; & \sum_{j \in \J} \omega_j \left(\langle C_j, \bar{X} \rangle + c_j^T \bar{x} + \gamma_j \right)  \label{sdp_da_daJ_2_c1} \\
=\;\; & \sum_{j \in \J} \omega_j \left( \langle C_j, \bar{X} - \bar{x} \bar{x}^T \rangle + \bar{x}^T C_j \bar{x} + c_j^T \bar{x} + \gamma_j \right) \label{sdp_da_daJ_2_c2} \\
=\;\; & \sum_{j \in \J} \omega_j\langle C_j, \bar{X} - \bar{x} \bar{x}^T \rangle = 2 \langle A^TA, \bar{X} - \bar{x} \bar{x}^T \rangle \label{sdp_da_daJ_2_c3}
\end{align}
\end{subequations}
where~\eqref{sdp_da_daJ_2_c2} follows from adding 
and subtracting $\bar{x}^T C_j \bar{x}$, the first equality in~\eqref{sdp_da_daJ_2_c3} follows from~(R\ref{recipe:step1}), the second equality in~\eqref{sdp_da_daJ_2_c3} from Proposition~\ref{prop:null_quad} and the condition in (ii). 
Thus $(\bar{x},\bar{y},\bar{X}) \in {\cal F}_{\text{SDPda}}$ proving the claim in (ii).
\end{proof}

Now, we are ready to present the main result of this section.

\begin{theorem}
Suppose that $\J$ is chosen such that (R\ref{recipe:step1}) holds. Let $\mu_{\emph{\text{SDPda}}}$, $\mu_{\textnormal{SICPda}\infty}$, and $\mu_{\emph{\text{SDPdaJ}}}$ be the optimal objective function values in~\eqref{sdp_da}, \eqref{sicp_da_inf}, and~\eqref{sdp_daJ}, respectively. Then, the following hold: \begin{enumerate}[(i)]
    \item $\mu_{\emph{\text{SDPdaJ}}} \leq \mu_{\emph{\text{SDPda}}}$
    \item If the assumption in Theorem~\ref{sdp_sicp_equivalence} holds, then $\mu_{\emph{\text{SDPdaJ}}} \leq \mu_{\textnormal{SICPda}\infty}$.
\end{enumerate}
\end{theorem}
\begin{proof}
The claim in (i) follows from Proposition~\ref{prop:sdp_relation_Z_J}. If the assumption in Theorem~\ref{sdp_sicp_equivalence} holds, then $\mu_{\emph{\textnormal{SDPda}}} = \mu_{\textnormal{SICPda}\infty}$ and the claim in (ii) follows from (i).
\end{proof}

\subsection{Cutting Surface Algorithm}
\label{sec:cutting_surface_algorithm}

By replacing $\D_{\alpha}$ in~\eqref{sicp_das_cuts} with a set $\D_{\alpha}^{(k)}$ of finite dimension, we can devise an iterative cutting surface algorithm which allows us to derive convex QCP relaxations for~\eqref{problem_statement}. At the $k$-th iteration of this algorithm, the following relaxation is solved:
\begin{subequations}
\label{qcp_das}
\begin{align}
 \underset{(x,y) \in {\cal F}, v}{\text{{min}}}\;\; & v + q^T x \label{qcp_das_obj} \\
\st\;\;   & v \geq x^T \left( Q + \text{{diag}}(d) \right) x - d^T y, \; \forall d \in \D_{\alpha}^{(k)}. \label{qcp_das_cuts}             
\end{align}
\end{subequations}
This iterative approach is described in Algorithm~\ref{alg:cutting_surface_algorithm}. At each iteration of this algorithm, a separation problem is solved in order to construct a new quadratic cut of the form~\eqref{qcp_das_cuts}. As the number of quadratic cuts increases, the resulting QCP relaxations become tighter. The algorithm terminates when either the maximum number of iterations MaxNC is reached or the new quadratic cut derived at iteration $k$ does not violate the solution of the QCP relaxation constructed at iteration $k-1$.

Note that the parameter $\alpha$ is fixed during the execution of this algorithm. Proposition~\ref{sicp_da_tightest_relaxation} suggests that we should select a large value of $\alpha$ in order to improve the bound given by~\eqref{qcp_das}. We describe a procedure to determine such a value of $\alpha$ in \S\ref{implementation}.

\begin{algorithm}
\caption{A cutting surface procedure to derive QCP relaxations for~\eqref{problem_statement}}
\label{alg:cutting_surface_algorithm}
\begin{algorithmic}
\STATE \textbf{Input}: $Q$, $q$, $A$, $b$, and algebraic expressions for $l(\cdot)$ and $u(\cdot)$.
\STATE \textbf{Output}: A lower bound $\mu_{\text{QCPda}}$ on the optimal solution of~\eqref{problem_statement}.
\IF{$m = 0$}
\STATE Set $\alpha = 0$
\ELSE
\STATE Choose a positive value of $\alpha$ according to the procedure described in \S\ref{implementation}.
\ENDIF
\STATE Set $\D_{\alpha}^{(0)} = \{d^{(0)}\}$, where $d^{(0)} \in \R^n$ is a perturbation for which~\eqref{qcp_das} is convex.
\STATE Solve~\eqref{qcp_das}. Let $(\bar{x}, \bar{y}, \bar{v})$ be an optimal solution to this relaxation.
\STATE Set $\mu_{\text{QCPda}} = \bar{v} + q^T \bar{x}$.
\FOR{$k = 1$ to $\text{MaxNC}$}
\STATE Solve a separation problem to find a new perturbation $d^{(k)}$. 
\IF{the quadratic cut~\eqref{qcp_das_cuts} with $d = d^{(k)}$ violates $(\bar{x}, \bar{y}, \bar{v})$}
\STATE $\D_{\alpha}^{(k)} \leftarrow \D_{\alpha}^{(k-1)} \cup \{d^{(k)}\}$
\STATE Solve~\eqref{qcp_das}. Let $(\bar{x}, \bar{y}, \bar{v})$ be an optimal solution to this relaxation.
\STATE Set $\mu_{\text{QCPda}} = \bar{v} + q^T \bar{x}$.
\ELSE
\STATE Terminate
\ENDIF
\ENDFOR
\end{algorithmic}
\end{algorithm}

Another interesting observation about Algorithm~\ref{alg:cutting_surface_algorithm} is that the parameter $\alpha$ only appears in the separation problem (see \S\ref{separation_problem_analysis} for details), since the term $\alpha \|Ax - b\|^2$ is not included in~\eqref{qcp_das_cuts}. This is particularly advantageous because it allows us to preserve the sparsity pattern of $Q$ in the quadratic constraints~\eqref{qcp_das_cuts}. In addition, by dropping the term $\alpha \|Ax - b\|^2$ from~\eqref{qcp_das_cuts}, we prevent the relaxation from becoming ill-conditioned for large values of $\alpha$.

In order to construct the first relaxation of Algorithm~\ref{alg:cutting_surface_algorithm}, we need to specify an initial perturbation $d^{(0)}$. For simplicity, we set $d^{(0)} = \mu  \mathbbm{1}$, where $\mu \in \R_{\geq 0}$. For this choice of $d^{(0)}$, \eqref{qcp_das} becomes:
\begin{subequations}
\label{qcp_da_init}
\begin{align}
 \underset{(x,y) \in {\cal F}, v}{\text{{min}}}\;\; & v + q^T x \\
\st\;\;   & v \geq x^T \left( Q + \mu I_n \right) x - \mu \mathbbm{1}^T y.    
\end{align}
\end{subequations}
In order to select $\mu$, we consider two cases depending on the value of $m$:
\begin{enumerate}[(i)]
\item If $m = 0$, \eqref{problem_statement} is an unconstrained optimization problem. We run Algorithm~\ref{alg:cutting_surface_algorithm} only if $Q$ is indefinite and set $\mu = -\lambda_{\text{min}} (Q)$. It is easy to verify that this choice of $\mu$ renders $Q + \mu I_n$ positive semidefinite, thus ensuring the convexity of~\eqref{qcp_da_init}.

\item If $m > 0$, \eqref{problem_statement} contains at least one equality constraint. We run Algorithm~\ref{alg:cutting_surface_algorithm} only if $Z^T Q Z$ is indefinite and set $\mu = -\lambda_{\text{min}} (Z^T Q Z)$. It is simple to check that, for this choice of $\mu$, \eqref{qcp_da_init} is a convex problem. To this end, note that the projection of~\eqref{qcp_da_init} onto the nullspace of A can be obtained from~\eqref{sicp_das_proj} by considering a single quadratic constraint in~\eqref{sicp_das_proj_c1} and setting $d = \mu \mathbbm{1}$. The resulting problem is convex when $Z^T \left(Q + \mu I_n \right) Z \succcurlyeq 0$. It is easy to verify that our choice of $\mu$ satisfies this condition.
\end{enumerate}
Since~\eqref{qcp_da_init} contains a single quadratic constraint and $\mu \geq 0$, we can eliminate the variables $y$ and $v$, and rewrite this QCP as the following quadratic program:
\begin{equation}
\label{qp_init}
\begin{array}{cl}
 \underset{x \in {\cal X}}{\text{{min}}}\;\; & x^T \left( Q + \mu I_n \right) x + q^T x - \mu \sum\limits_{i = 1}^{n} u_i(x_i).
\end{array}
\end{equation}
We refer to the relaxations obtained by setting $\mu = - \lambda_{\text{min}} (Q)$ and $\mu = -\lambda_{\text{min}} (Z^T Q Z)$ in~\eqref{qp_init} as the~\textit{eigenvalue relaxation} and the~\textit{eigenvalue relaxation in the nullspace of $A$}, respectively (see Section 3 in~\cite{nrs:20} for details). In (ii), we could use the same initial perturbation as in (i). However, as shown in~\cite{nrs:20}, the perturbation given in (ii) can lead to a tighter initial bound.

Observe that the spectral relaxations introduced in~\cite{nrs:20} are derived through a uniform diagonal perturbation of $Q$, i.e., by considering a perturbation vector $d \in \R^n$ whose entries are identical to each other. By contrast, the QCP relaxations of Algorithm~\ref{alg:cutting_surface_algorithm} are derived through a nonuniform diagonal perturbation of $Q$, i.e., by considering a perturbation vector $d \in \R^n$  whose entries are allowed to differ from each other. Note also that, since the bound provided by~\eqref{qcp_da_init} is equal to the bound given by~\eqref{qp_init}, by construction, the QCP relaxations solved during the execution in Algorithm~\ref{alg:cutting_surface_algorithm} are at least as tight at the spectral relaxations considered in~\cite{nrs:20}.

\section{Analysis and regularization of the separation problem}
\label{separation_problem_analysis}

We start this section by presenting the separation problem solved in Algorithm~\ref{alg:cutting_surface_algorithm}. Let $(\bar{x}, \bar{y}, \bar{v})$ be an optimal solution to the relaxation~\eqref{qcp_das}. Then, in order to construct a quadratic inequality of the form~\eqref{qcp_das_cuts} that is maximally violated by $(\bar{x}, \bar{y}, \bar{v})$, we can solve the following optimization problem:
\begin{equation}
\label{separation_da_sup}
\begin{aligned}
\underset{d \in \D_{\alpha}}{\text{sup}}\;\; & \sum\limits_{i = 1}^{n} \left(\bar{x}_i^2 - \bar{y}_i \right) d_i.
\end{aligned}
\end{equation}
This SDP is parametrized by the value of $\alpha$ determined at the beginning of Algorithm~\ref{alg:cutting_surface_algorithm}. Observe also that the separation problem considered in~\cite{d:16} is a particular instance of~\eqref{separation_da_sup}, obtained by setting $\alpha = 0$ in~\eqref{separation_da_sup}. For the remainder of this section, we will cast~\eqref{separation_da_sup} as:
\begin{equation}
\label{separation_da}
\begin{aligned}
\underset{d \in \D_{\alpha}}{\text{inf}}\;\; & \eta^T d
\end{aligned}
\end{equation}
where $\eta_i := \bar{y}_i - \bar{x}_i^2, \; \forall i \in [n]$. As shown in~\cite{d:16}, the attainment of the infimum in~\eqref{separation_da} is not guaranteed, and may depend on the problem data. We illustrate this behavior through the following example.

\begin{example}
\label{example_attainment}

Let $Q = \begin{bsmallmatrix} 0 & \;\;\;2 \\ 2 & \;-1\end{bsmallmatrix}$, $A =\begin{bsmallmatrix} 0 & 1 \end{bsmallmatrix}$, and $\alpha = 1$ in~\eqref{separation_da}. Consider the following cases:

\begin{enumerate}[(i)]

\item $\bar{x} = \bar{y} = \begin{bsmallmatrix} 0.5 & \; 0.5 \end{bsmallmatrix}^T$. In this case, the infimum in~\eqref{separation_da} is attained for $d_1^* = d_2^* = 2$.

\item $\bar{x} = \bar{y} = \begin{bsmallmatrix} 0.4 & \; 0 \end{bsmallmatrix}^T$. In this case, the infimum in~\eqref{separation_da} cannot be attained since it occurs as $d_1 \rightarrow 0$ and $d_2 \rightarrow \infty$.

\end{enumerate}

\end{example}

To further analyze the attainment of~\eqref{separation_da}, we construct the dual of this SDP. Let $Y \in \Sy^n$ be the matrix of dual variables associated with the semidefinite constraint $Q + \text{diag}(d) + \alpha A^T A \succcurlyeq 0$. Then, the dual of~\eqref{separation_da} is:
\begin{subequations}
\label{separation_da_dual}
\begin{align}
\underset{Y \succcurlyeq 0}{\text{sup}}\;\; & -\langle Q + \alpha A^T A, Y \rangle \\
  \st\;\;      & Y_{ii} = \eta_i, \; \forall i \in [n]. \label{separation_da_dual_c1}
\end{align}
\end{subequations}
Let $P:= \{ i \in [n] \, : \, \bar{y}_i = \bar{x}_i^2 \}$. From~\eqref{separation_da_dual_c1}, it follows that $Y_{ii} = 0, \; \forall i \in P$. Hence, if $P \neq \emptyset$, \eqref{separation_da_dual} does not admit a strictly feasible solution and, as a result, strong duality may not hold for the primal-dual pair \eqref{separation_da},\eqref{separation_da_dual}. In Example~\ref{example_attainment}, we have $P = \emptyset$ for (i), and $P = \{2\}$ for (ii).

In the separation step of Algorithm~\ref{alg:cutting_surface_algorithm}, we do not need to solve~\eqref{separation_da} to optimality, but rather derive quadratic cuts that can be used to tighten a relaxation of the form~\eqref{qcp_das}. As a result, we can replace~\eqref{separation_da} with a regularized separation problem constructed in a way such that its optimum is always attained. One option is to regularize~\eqref{separation_da} as discussed in~\cite{d:16}. To this end, we can add to the objective function of~\eqref{separation_da} the term $\lambda \sum_{i = 1}^{n} {[d_i]}_{+}$, where $\lambda = \sum_{i = 1}^{n} (\bar{y}_i - \bar{x}_i^2)$, and ${[d_i]}_{+}$ is equal to $d_i$ if $d_i > 0$, and $0$ otherwise. This leads to the following regularized separation problem:
\begin{equation}
\label{separation_da_nonsmooth_reg}
\begin{aligned}
\underset{d \in \D_{\alpha}}{\text{inf}}\;\; & \eta^T d + \lambda \sum\limits_{i = 1}^{n} {[d_i]}_{+}.
\end{aligned}
\end{equation}
It is simple to show that the infimum in this SDP is always attained (see Proposition 3 in~\cite{d:16} for details). The parameter $\lambda$ is always positive unless the current relaxation is exact.

In this paper, we propose an alternative regularization for~\eqref{separation_da}. We modify this problem by adding to the objective function the quadratic term $\rho  d^T d$, where $\rho$ is a positive scalar. The resulting regularized separation problem is:
\begin{equation}
\label{separation_da_smooth_reg}
\begin{aligned}
\underset{d \in \D_{\alpha}}{\text{inf}}\;\; & \eta^T d + \rho  d^T d.
\end{aligned}
\end{equation}
As we show in the following proposition, this regularization also gives rise to a separation problem for which the optimum is always attained.

\begin{proposition}
\label{smooth_reg_attainment}
Let $\rho > 0$ in~\eqref{separation_da_smooth_reg}. Then, the optimal solution to the semidefinite program~\eqref{separation_da_smooth_reg} is always attained at some finite point.
\end{proposition}

\begin{proof}
This proof relies on strong duality holding for~\eqref{separation_da_smooth_reg} and its dual. Denote by $Y \in \Sy^n$ the matrix of dual variables associated with the semidefinite constraint $Q + \text{diag}(d) + \alpha A^T A \succcurlyeq 0$. Then, the dual of~\eqref{separation_da_smooth_reg} is:
\begin{equation}
\label{separation_da_smooth_reg_dual}
\begin{aligned}
\underset{Y \succcurlyeq 0}{\text{sup}}\;\; & -\langle Q + \alpha A^T A, Y \rangle - \dfrac{1}{4 \rho} \sum\limits_{i = 1}^{n} {\left( Y_{ii} - \eta_i \right)}^2.
\end{aligned}
\end{equation}
Now, let $\bar{d} = \mu \mathbbm{1}$ and $\bar{Y} = I_n$, where $\mu > -\min(0, \lambda_{\text{min}} (Q + \alpha A^T A))$. Clearly, $\bar{d}$ and $\bar{Y}$ are strictly feasible in~\eqref{separation_da_smooth_reg} and~\eqref{separation_da_smooth_reg_dual}, respectively. Hence, strong duality holds, and both SDPs attain their optimal solutions. 
\end{proof}

\section{Solution of the regularized separation problem}
\label{coordinate_minimization_algorithm}

In this section, we describe the algorithm that we use to solve the regularized separation problem proposed in \S\ref{separation_problem_analysis}. This algorithm is a modification of the barrier coordinate minimization method introduced in~\cite{d:16}. To solve~\eqref{separation_da_smooth_reg}, our algorithm operates on the following penalized log-det problem:
\begin{equation}
\label{separation_da_smooth_reg_log_det}
\begin{array}{cl}
\underset{d \in \R^n}{\text{inf}}\;\; & f(d; \sigma) := G(d) - \sigma \text{log-det}\left(Q + \text{diag}(d) + \alpha A^T A \right) \\
  \st\;\;      & Q + \text{diag}(d) + \alpha A^T A \succ 0
\end{array}
\end{equation}
where $G(d) = \sum_{i = 1}^{n} g_i(d_i)$, $g_i(d_i) = \eta_i d_i + \rho d_i^2, \; \forall i \in [n]$, and $\sigma$ is a positive penalty parameter. The optimality condition for~\eqref{separation_da_smooth_reg_log_det} can be expressed as:
\begin{equation}
\label{optimality_condition}
\begin{array}{cl}
\nabla f(d; \sigma) = 0, \;\;\; Q + \text{diag}(d) + \alpha A^T A \succ 0
\end{array}
\end{equation}
where the gradient of $f(d; \sigma)$ has the form:
\begin{equation}
\label{gradient_f}
\begin{array}{cl}
\nabla f(d; \sigma) = \nabla G(d)  - \sigma \text{diag} \left( {\left[ Q + \text{diag}(d) + \alpha A^T A \right]}^{-1} \right)
\end{array}
\end{equation}
with $\nabla G(d)_i = \eta_i + 2 \rho d_i, \; \forall i \in [n]$. At each iteration of this algorithm, we update a feasible vector $\bar{d}$ and an inverse matrix $V:= {\left[ Q + \text{diag}(\bar{d}) + \alpha A^T A \right]}^{-1}$. Based on the optimality condition in~\eqref{optimality_condition}, we perform coordinate minimization by choosing an index $i$ which corresponds to the entry of $\nabla f(\bar{d}; \sigma)$ with the largest magnitude:
\begin{equation}
\label{index_choice}
\begin{array}{cl}
i = \text{arg}\underset{j = 1, \dots, n}{\text{max}} \left\lbrace \left| {\nabla f(\bar{d}; \sigma)}_j \right| \right\rbrace.
\end{array}
\end{equation}
This choice of $i$ leads to the following one-dimensional minimization problem:
\begin{equation}
\label{one_dimensional_problem}
\begin{array}{cl}
\Delta d_i^* \in \text{arg}\underset{\Delta d_i}{\text{min}} \left\lbrace f(\bar{d} + \Delta d_i e_i; \sigma)   \, : \, Q + \text{diag}(\bar{d} + \Delta d_i e_i) + \alpha A^T A \succ 0 \right\rbrace.
\end{array}
\end{equation}
As we show in the next proposition, it is possible to find a closed-form expression for the optimal solution of~\eqref{one_dimensional_problem}.

\begin{proposition}
\label{closed_form_solution}
The optimal solution to the one-dimensional problem~\eqref{one_dimensional_problem} is:
\begin{equation}
\label{delta_di}
\begin{array}{cl}
\Delta d_i^* = -(\phi_i + \tau_i) + \sqrt{ {(\phi_i - \tau_i)}^2 + \kappa }
\end{array}
\end{equation}
where $\phi_i = 1/(2 V_{ii})$, $\tau_i = (\eta_i + 2 \rho \bar{d}_i )/ (4\rho)$ and $\kappa = \sigma / (2 \rho)$.
\end{proposition}

\begin{proof}

At the optimal solution of~\eqref{one_dimensional_problem} the following holds:
\begin{equation}
\label{optimality_condition_1D_b}
\begin{array}{cl}
\dfrac{\partial f(\bar{d} + \Delta d_i e_i; \sigma)}{\partial  \Delta d_i} = \eta_i + 2 \rho (\bar{d}_i + \Delta d_i) - \dfrac{\sigma V_{ii}}{ 1 + \Delta d_i V_{ii}} = 0.
\end{array}
\end{equation}
By solving for $\Delta d_i$ in~\eqref{optimality_condition_1D_b}, we obtain the roots ${\Delta d_i^*}^{(\pm)} = -(\phi_i + \tau_i) \pm \sqrt{ {(\phi_i - \tau_i)}^2 + \kappa }$. It is easy to show that ${\Delta d_i^*}^{(+)}$ is the only one of these two solutions that is feasible in~\eqref{one_dimensional_problem}. By applying Lemma 1 from~\cite{d:16}, we have:
\begin{equation}
\label{delta_di_feasibility_condition}
\begin{aligned}
Q + \text{diag}(\bar{d} + \Delta d_i e_i) + \alpha A^T A \succ 0 \iff \Delta d_i > -1 / V_{ii}.
\end{aligned}
\end{equation}
Therefore, for ${\Delta d_i^*}^{(+)}$ to be feasible in~\eqref{one_dimensional_problem} we must have $z + \sqrt{ z^2 + \kappa } > 0$, where $z = (\phi_i - \tau_i)$. It is simple to check that this condition is always satisfied. To this end, note that $\kappa > 0$. This implies $z + \sqrt{ z^2 + \kappa } > z + |z| \geq 0, \; \forall z \in \R$. Using a similar analysis, we can show that ${\Delta d_i^*}^{(-)}$ is infeasible in~\eqref{one_dimensional_problem}.
\end{proof}

After solving the one-dimensional problem~\eqref{one_dimensional_problem}, we update $\bar{d}$ as:
\begin{equation}
\label{update_d}
\begin{aligned}
\bar{d} \leftarrow \bar{d} + \Delta d_i^* e_i \\
\end{aligned}
\end{equation}
and update $V$ using the Sherman-Morrison formula:
\begin{equation}
\label{update_V}
\begin{aligned}
V \leftarrow V - \dfrac{\Delta d_i^*  V_{\cdot i} {V_{\cdot i}}^T}{1 + \Delta d_i^*  V_{ii}}.
\end{aligned}
\end{equation}

In our numerical experiments, we noticed that, for very small values of $\rho$, some of the entries of $\bar{d}$ become very large after performing the update in~\eqref{update_d}. This is not surprising because: (i) for very small values of $\rho$,~\eqref{separation_da_smooth_reg} exhibits a similar behavior to~\eqref{separation_da}, and (ii) as discussed in \S\ref{separation_problem_analysis}, the finite attainment of~\eqref{separation_da} is not guaranteed. To address this issue, we propose an adaptive strategy in order to adjust the value of $\rho$ used in~\eqref{separation_da_smooth_reg}. At a given iteration of our algorithm, after performing the update in~\eqref{update_d}, we determine the entry of $\bar{d}$ with the largest magnitude, i.e.:
\begin{equation}
\label{d_max}
\begin{array}{cl}
\bar{d}_{\text{max}} = \underset{j = 1, \dots, n}{\text{max}} \left\lbrace \left| \bar{d}_j \right| \right\rbrace.
\end{array}
\end{equation}
If $\bar{d}_{\text{max}}$ is at least an order of magnitude larger than $\mu :=-\lambda_{\text{min}} (Q + \alpha A^T A)$, we increase $\rho$ by multiplying it by a factor $\rho_{\text{upd}}$, and restart our algorithm with this new value of $\rho$. In practice, this adaptive strategy only requires a few restarts before finding a suitable value of $\rho$. For the first run, we set $\rho = \rho_{\text{init}}$.

Once~\eqref{separation_da_smooth_reg_log_det} has been solved within a given precision, we update the penalty parameter $\sigma$ according to the following condition:
\begin{equation}
\label{update_sigma}
\begin{aligned}
\sigma \leftarrow \max\{\sigma_{\text{min}}, \sigma_{\text{upd}}  \sigma \} \;\;\;\; \text{if} \;\;\;\; \dfrac{ {\Vert \nabla f(\bar{d}; \sigma) \Vert}_2 }{ {\Vert \eta \Vert}_2 } \leq \epsilon_{\text{upd}}
\end{aligned}
\end{equation}
where $\nabla f(\bar{d}; \sigma)$  is used as a measure of optimality. We check the relative improvement in the objective function of~\eqref{separation_da_smooth_reg} every $\omega_{\text{check}}  n$ iterations, and terminate our algorithm if this relative improvement is smaller than $\epsilon_{\text{check}}$.  

Our coordinate minimization strategy is summarized in Algorithm~\ref{alg:coordinate_minimization_smooth}. If $Q + \alpha A^T A$ is positive semidefinite, then $Z^T Q Z$ is also positive semidefinite and~\eqref{problem_statement} is convex when restricted to the nullspace of $A$. In this case, it suffices to  solve the continuous relaxation of~\eqref{problem_statement} in order to obtain a lower bound. As a result, the separation procedure outlined in Algorithm~\ref{alg:coordinate_minimization_smooth} is only used if $Q + \alpha A^T A$ is indefinite. We start Algorithm~\ref{alg:coordinate_minimization_smooth} with an initial perturbation $\hat{d} = -1.5 \lambda_{\text{min}} (Q + \alpha A^T A)  \mathbbm{1}$. We set $\text{MaxIter} = 500  n$, $\sigma_{\text{min}} = 10^{-5}$, $\sigma_{\text{upd}} = 0.8$, $\epsilon_{\text{upd}} = 0.03$, $\omega_{\text{check}} = 10$ and $\epsilon_{\text{check}} = 10^{-4}$. We initialize $\rho_{\text{init}}$ as:
\begin{equation}
\label{rho_init}
\begin{aligned}
\rho_{\text{init}} = 10^{-4}  \dfrac{10^{4  \floor{\log_{10} (\delta_{\max})}}}{\max\{1, \floor{Q_{\max} / 100}  Q_{\max} \}}
\end{aligned}
\end{equation}
where $Q_{\max}$ and $\delta_{\max}$ are given by:
\begin{equation}
\label{rho_init_parameters}
\begin{aligned}
Q_{\max} = \underset{i = 1, \dots, n, \; j = i, \dots, n}{\text{max}} \left\lbrace \left| Q_{ij} \right| \right\rbrace, \;\;\; \delta_{\max} = \underset{i = 1, \dots, n}{\text{max}} \left\lbrace U_i - L_i \right\rbrace
\end{aligned}
\end{equation}
and set $\rho_{\text{upd}} = 10$. The initial value of  $\sigma_{\text{init}}$ is determined as:
\begin{equation}
\label{sigma_formula}
\begin{aligned}
\sigma_{\text{init}} = {{\text{median}} \left\lbrace \left| \dfrac{\eta_i + 2 \rho \hat{d}_i}{V_{ii}} \right| \right\rbrace}_{i = 1}^{n}.
\end{aligned}
\end{equation}

\begin{algorithm}
\caption{Barrier coordinate minimization algorithm used to solve the smooth regularized separation problem~\eqref{separation_da_smooth_reg}}
\label{alg:coordinate_minimization_smooth}
\begin{algorithmic}[1]
\STATE \textbf{Input}: $Q$, $A$, $\alpha$, and an optimal solution $(\bar{x}, \bar{y}, \bar{v})$ to~\eqref{qcp_das}.
\STATE \textbf{Output}: A vector $\bar{d}$ that solves~\eqref{separation_da_smooth_reg}.
\STATE Set $\rho = \rho_{\text{init}}$, where $\rho_{\text{init}}$ is determined using~\eqref{rho_init}.
\STATE Set $\bar{d} = \hat{d}$, where $\hat{d} = 1.5 \mu  \mathbbm{1}$ and $\mu = -\lambda_{\text{min}} (Q + \alpha A^T A)$.
\STATE Set $\sigma = \sigma_{\text{init}}$, where $\sigma_{\text{init}}$ is calculated according to~\eqref{sigma_formula}.
\STATE Calculate $V = {\left[ Q + \text{diag}(\bar{d}) + \alpha A^T A \right]}^{-1}$ and set $k = 0$.
\WHILE{$(k < \text{MaxIter})$}
\STATE Update $k \leftarrow k + 1$.
\STATE Determine an index $i$ according to~\eqref{index_choice} and calculate $\Delta d_i^*$ using~\eqref{delta_di}.
\STATE Update $\bar{d}$ according to~\eqref{update_d} and determine $\bar{d}_{\max}$ using~\eqref{d_max}.
\IF{$(\bar{d}_{\max} > 10  \mu)$}
\STATE Update $\rho \leftarrow \rho_{\text{upd}}  \rho $ and \textbf{goto} 4.
\ENDIF
\STATE Update $V$ according to~\eqref{update_V}.
\STATE Adjust $\sigma$ according to~\eqref{update_sigma}.
\IF{$(k \bmod (\omega_{\text{check}}  n) = 0)$}
\STATE Terminate if the improvement in the objective of~\eqref{separation_da_smooth_reg} is smaller than $\epsilon_{\text{check}}$.
\ENDIF
\ENDWHILE
\end{algorithmic}
\end{algorithm}

Even though Algorithm~\ref{alg:coordinate_minimization_smooth} is a variant of the barrier coordinate minimization algorithm introduced in~\cite{d:16}, there are two key differences between these two algorithms. First, unlike the algorithm presented in~\cite{d:16}, Algorithm~\ref{alg:coordinate_minimization_smooth} does not rely on nonsmooth optimization techniques because the objective function of our regularized separation problem~\eqref{separation_da_smooth_reg} is smooth. Second, the regularization parameter $\lambda$ used in~\eqref{separation_da_nonsmooth_reg} is fixed throughout the execution of the algorithm proposed in~\cite{d:16}. By contrast, in Algorithm~\ref{alg:coordinate_minimization_smooth}, we adaptively adjust the regularization parameter $\rho$ used in~\eqref{separation_da_smooth_reg}. As we demonstrate in \S\ref{sec:results_relaxations}, because of this adaptive strategy, the quadratic cuts derived by solving~\eqref{separation_da_smooth_reg} with our algorithm lead to significantly tighter relaxations than the quadratic cuts obtained through the solution of~\eqref{separation_da_nonsmooth_reg} with the algorithm proposed in~\cite{d:16}. For the sake of completeness, in \S\ref{appendix}, we provide a detailed description of the barrier coordinate minimization algorithm introduced in~\cite{d:16}.

\section{Implementation in a branch-and-bound algorithm}
\label{implementation}

BARON's portfolio of relaxations consists of linear, quadratic, nonlinear, and mixed-integer linear programming relaxations~\cite{ks:18,ks19,nrs:20,ts:comp:04}. As part of our implementation, we have expanded this portfolio by adding the convex relaxations described in \S\ref{sec:cutting_surface_algorithm}. These new relaxations are only used if the original model supplied to BARON is of the form~\eqref{problem_statement}.

At the root node of the branch-and-bound tree, we solve QCP relaxations of the form~\eqref{qcp_das} by running Algorithm~\ref{alg:cutting_surface_algorithm} with the maximum number of iterations $\text{MaxNC}$ set to $20$. In each iteration of this algorithm, we generate quadratic cuts of the form~\eqref{qcp_das_cuts} by solving the regularized separation problem~\eqref{separation_da_smooth_reg} with Algorithm~\ref{alg:coordinate_minimization_smooth}. As indicated in \S\ref{sec:cutting_surface_algorithm}, for $m> 0$, the initial perturbation used in Algorithm~\ref{alg:cutting_surface_algorithm} is set as $d^{(0)} = \mu  \mathbbm{1}$, where $\mu = -\lambda_{\text{min}} (Z^T Q Z)$. In our previous paper~\cite{nrs:20}, we showed that it is possible to approximate $\lambda_{\text{min}} (Z^T Q Z)$ without having to explicit compute the basis $Z$. In particular, we proved that:
\begin{equation}
\label{eigenvalue_approximation}
\begin{aligned}
\lim_{\alpha \to\infty} \lambda_{{\text{min}}} (Q, I_n + \alpha A^T A) = \min(0, \lambda_{{\text{min}}} (Z^T Q Z)).
\end{aligned}
\end{equation}
Using this result, we set $\mu = \mu(\alpha) := -\lambda_{{\text{min}}} (Q, I_n + \alpha A^T A)$. From~\eqref{eigenvalue_approximation}, it follows that, for a sufficiently large value of $\alpha$, $\mu(\alpha)$ will converge to $0$ if~\eqref{problem_statement} is convex when restricted to the nullspace of $A$, or $-\lambda_{\text{min}} (Z^T Q Z)$ otherwise. To find such value of $\alpha$, we follow the iterative procedure presented in Section 5 of~\cite{nrs:20}. This is the same value of $\alpha$ that we use in Algorithm~\ref{alg:cutting_surface_algorithm} and the regularized separation problem~\eqref{separation_da_smooth_reg}. Note that this value of $\alpha$ determined at the root-node is used throughout the entire branch-and-bound tree.

Algorithm~\ref{alg:cutting_surface_algorithm} is only called at the root-node of the branch-and-bound tree. At nodes other than the root-node, instead of solving QCP relaxations, we solve QP relaxations of the form:
\begin{equation}
\label{qp_child_nodes}
\begin{array}{cl}
 \underset{(x,y) \in {\cal F}}{\text{{min}}}\;\; & x^T \left( Q + \text{diag}(d) \right) x + q^T x - d^T y
\end{array}
\end{equation}
where $d \in \D_{\alpha}$. We proceed as follows:

\begin{enumerate}[(i)]

\item We solve an initial relaxation of the form~\eqref{qp_child_nodes} by setting $d = d^{\text{parent}}$, where $d^{\text{parent}}$ is a diagonal perturbation originating from the parent node. Let $(\bar{x}, \bar{y})$ be the optimal solution of this initial QP relaxation and denote by $\bar{\mu}_{\text{QP}}$ its optimal objective function value.

\item  We use this relaxation solution to construct a new perturbation $d^{\text{new}}$ by solving the regularized separation problem~\eqref{separation_da_smooth_reg} with Algorithm~\ref{alg:coordinate_minimization_smooth}.

\item If $\bar{\mu}_{\text{QP}} - q^T \bar{x} < \bar{x}^T \left( Q + \text{diag}(d^{\text{new}}) \right) \bar{x} - {(d^{\text{new}})}^T \bar{y}$, we solve a second relaxation of the form~\eqref{qp_child_nodes} by setting $d = d^{\text{new}}$. Let $(\hat{x}, \hat{y})$ be an optimal solution of this relaxation and denote by $\hat{\mu}_{\text{QP}}$ its optimal objective function value. If $\hat{\mu}_{\text{QP}} \geq \bar{\mu}_{\text{QP}}$ (resp. $\hat{\mu}_{\text{QP}} < \bar{\mu}_{\text{QP}})$, we use the bound $\hat{\mu}_{\text{QP}}$ (resp. $\bar{\mu}_{\text{QP}}$) and pass $d^{\text{new}}$ (resp. $d^{\text{parent}}$) to the descendant nodes of the current node.

\end{enumerate}

For the descendant nodes of the root-node, we set $d^{\text{parent}} = d^{\text{root}}$, with $d^{\text{root}}$ being a surrogate perturbation vector determined as:
\begin{equation}
\label{surrogate}
\begin{array}{cl}
d^{\text{root}} = \dfrac{1}{\sum_{i = 1}^{\text{NC}} \nu_i} \sum\limits_{i = 1}^{\text{NC}} \nu_i d^{(i)}, 
\end{array}
\end{equation}
where $\text{NC}$ is the number of quadratic cuts generated during the execution of Algorithm~\ref{alg:cutting_surface_algorithm} at the root-node, $d^{(i)}$ are the diagonal perturbations, and $\nu_i$ are the optimal Lagrange multipliers associated with the quadratic constraints of the last root-node relaxation of the form~\eqref{qcp_das}. Note that $\text{NC} \leq \text{MaxNC}$, since Algorithm~\ref{alg:cutting_surface_algorithm} might terminate before reaching the maximum number of iterations $\text{MaxNC}$.

The decision to solve QP relaxations instead of QCP relaxations at nodes other than the root node is motivated by two key observations. First, the convex QCP relaxations of the form~\eqref{qcp_das} are at least an order of magnitude more expensive than the QP relaxations of the form~\eqref{qp_child_nodes}. Second, often a single quadratic cut of the form~\eqref{qcp_das_cuts} leads to a significant bound improvement. As a result, there is little gain in running Algorithm~\ref{alg:coordinate_minimization_smooth} more than once. Since the first QP relaxation constructed at the descendant nodes always uses a diagonal perturbation originating from the parent node, the monotonicity of the bounds generated during the branch-and-bound search is guaranteed.

To solve the eigenvalue and generalized eigenvalue problems that arise during the construction of the relaxations discussed above, we use the subroutines included in the linear algebra library LAPACK~\cite{1999lapack}. At a given node of the branch-and-bound tree, we only consider the variables that have not been fixed in order to construct our relaxations. We solve the convex QCP relaxations with IPOPT and the convex QP relaxations with CPLEX. The relaxation solutions returned by these solvers are used at a given node only if they satisfy the KKT conditions. At nodes at which~\eqref{problem_statement} is convex, we do not use the relaxations described in this section, and solve instead a continuous relaxation of~\eqref{problem_statement} subject to the variable bounds of the current node.

When all the variables in~\eqref{problem_statement} are binary, $l_i(x_i) = u_i(x_i) = x_i, \; \forall i \in [n]$, and as a result, we can eliminate the $y$ variables from~\eqref{qcp_das} and~\eqref{qp_child_nodes}. For continuous and general integer variables, we use $l_i(x_i) = x_i^2, \; i \in [n]$. For general integer variables, this choice of $l_i(x_i)$ does not lead to the convex hull of $\C_i$, but it allows us to construct a convex outer-approximation for this set.

Our implementation relies on the dynamic relaxation selection strategy proposed in~\cite{nrs:20} in order to adjust the frequencies at which we solve polyhedral and quadratic relaxations during the branch-and-bound search. Moreover, if~\eqref{problem_statement} is a binary quadratic program, we use the spectral braching variable selection rule introduced in~\cite{nrs:20}. The QP relaxations~\eqref{qp_child_nodes} are only used during the branch-and-bound search if, at the root-node, Algorithm~\ref{alg:cutting_surface_algorithm} gives a tighter bound than BARON's LP relaxation. Otherwise, we disable these QP relaxations and utilize the spectral relaxations proposed in~\cite{nrs:20}.

\section{Computational results}
\label{computational_results}

In this section, we investigate the impact of the quadratic relaxations proposed in this paper on the performance of branch-and-bound algorithms. We start in \S\ref{sec:results_relaxations}, by showing the effectiveness of the regularization approach discussed in \S\ref{separation_problem_analysis}. Then, in \S\ref{sec:results_baron}, we demonstrate the benefits of the implementation described in \S\ref{implementation} on the performance of BARON. Finally, in \S\ref{sec:results_solvers}, we present a comparison between several state-of-the-art global optimization solvers.

Our experiments are conducted under GAMS 30.1.0 on a 64-bit Intel Xeon X5650 2.66GHz processor with a single-thread. For the experiments described in \S\ref{sec:results_relaxations}, we solve the QP relaxations with CPLEX 12.10, the QCP relaxations with IPOPT 3.12 and the SDP relaxations with MOSEK 9.1.9. For the experiments considered in \S\ref{sec:results_baron}--\ref{sec:results_solvers}, we consider the following global optimization solvers: ANTIGONE 1.1, BARON 20.4, COUENNE 0.5, CPLEX 12.10, GUROBI 9.0, LINDOGLOBAL 12.0 and SCIP 6.0. In this case we: (i) run all solvers with relative/absolute tolerances of 10\textsuperscript{-6} and a time limit of 500 seconds, and (ii) set the CPLEX option {\tt optimalitytarget} to 3 and the GUROBI option {\tt nonconvex} to 2 to ensure that these two solvers search for a globally optimal solution. We use default settings for other algorithmic parameters.

For our experiments, we consider a large test set consisting of 960 Cardinality Binary Quadratic Programs (CBQPs), 30 Quadratic Semi-Assignment Problems (QSAPs), 246 Box-Constrained Quadratic Programs (BoxQPs), and 315 Equality Integer Quadratic Programs (EIQPs). These test libraries are described in detail in~\cite{qp_miqp_tests}.

\subsection{Experiments with root-node relaxations}
\label{sec:results_relaxations}

In this section, we provide a numerical comparison between two versions of Algorithm~\ref{alg:cutting_surface_algorithm} which differ in the separation procedure used to derive the quadratic cuts of the form~\eqref{qcp_das_cuts}. We use the following notation for the relaxations considered in this comparison:

\begin{enumerate}[(i)]

\item EIG: Eigenvalue relaxation, obtained by setting $\mu = -\lambda_{\text{min}} (Q)$ in~\eqref{qp_init}.

\item EIGNS: Eigenvalue relaxation in the nullspace of $A$, obtained by setting $\mu = -\lambda_{\text{min}} (Z^T Q Z)$ in~\eqref{qp_init}.

\item SDPd: SDP relaxation~\eqref{sdp_d}.

\item SDPda: SDP relaxation~\eqref{sdp_da}.

\item QCPnsreg: QCP relaxation~\eqref{qcp_das}, where the quadratic cuts~\eqref{qcp_das_cuts} are obtained by solving~\eqref{separation_da_nonsmooth_reg} with the algorithm proposed in~\cite{d:16} (see Algorithm~\ref{alg:coordinate_minimization_nonsmooth} in \S\ref{appendix}).

\item QCPsreg: QCP relaxation~\eqref{qcp_das}, where the quadratic cuts~\eqref{qcp_das_cuts} are obtained by solving~\eqref{separation_da_smooth_reg} with Algorithm~\ref{alg:coordinate_minimization_smooth}.

\end{enumerate}

In our experiments, we run the two versions of Algorithm~\ref{alg:cutting_surface_algorithm} by setting the maximum number of iterations $\text{MaxNC}$ to 20. We first compare these relaxations by selecting one instance from each of the four libraries that are part of the test set. The results of this comparison are presented in Figures~\ref{fig:CBQP_example}--\ref{fig:EIQP_example}. In these figures, we plot the lower bounds of the QCP relaxations against the number of iterations, and use a dashed vertical line to indicate the iteration number at which each version of Algorithm~\ref{alg:cutting_surface_algorithm} terminates. We use horizontal lines to represent the lower bounds provided by the spectral and SDP relaxations. As seen in the figures, the quadratic cuts derived by solving~\eqref{separation_da_smooth_reg} with Algorithm~\ref{alg:coordinate_minimization_smooth} lead to significantly tighter QCP relaxations than the quadratic cuts obtained through the solution of~\eqref{separation_da_nonsmooth_reg} with the algorithm proposed in~\cite{d:16}. Under our approach, a few quadratic cuts are sufficient in order to obtain a good approximation of the lower bounds given by the SDP relaxations.

\begin{figure}[htp]
\centering
\subfloat[Selected CBQP instance.\label{fig:CBQP_example}]{%
  \includegraphics[scale=0.12]{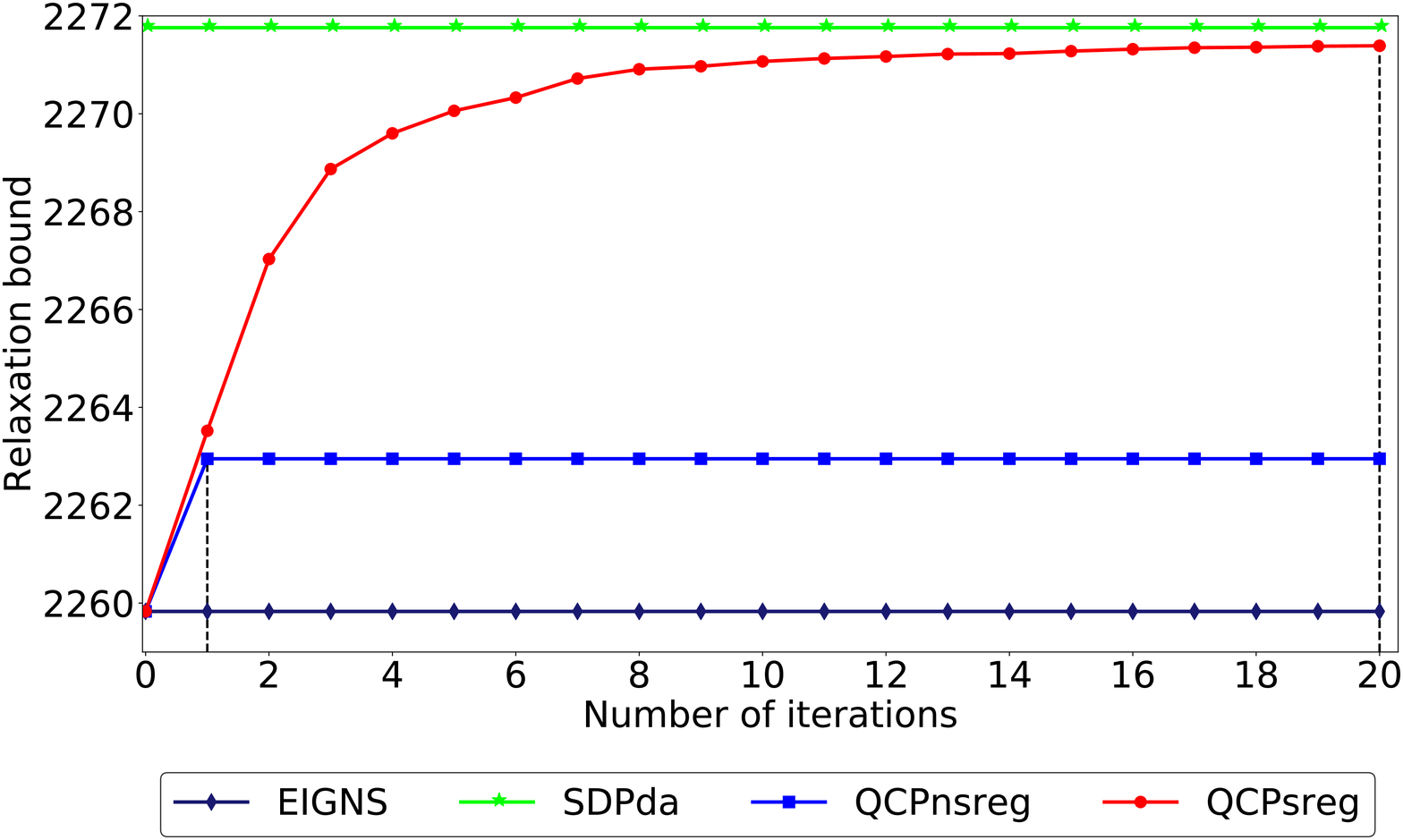}%
}
\subfloat[Selected QSAP instance.\label{fig:QSAP_example}]{%
  \includegraphics[scale=0.12]{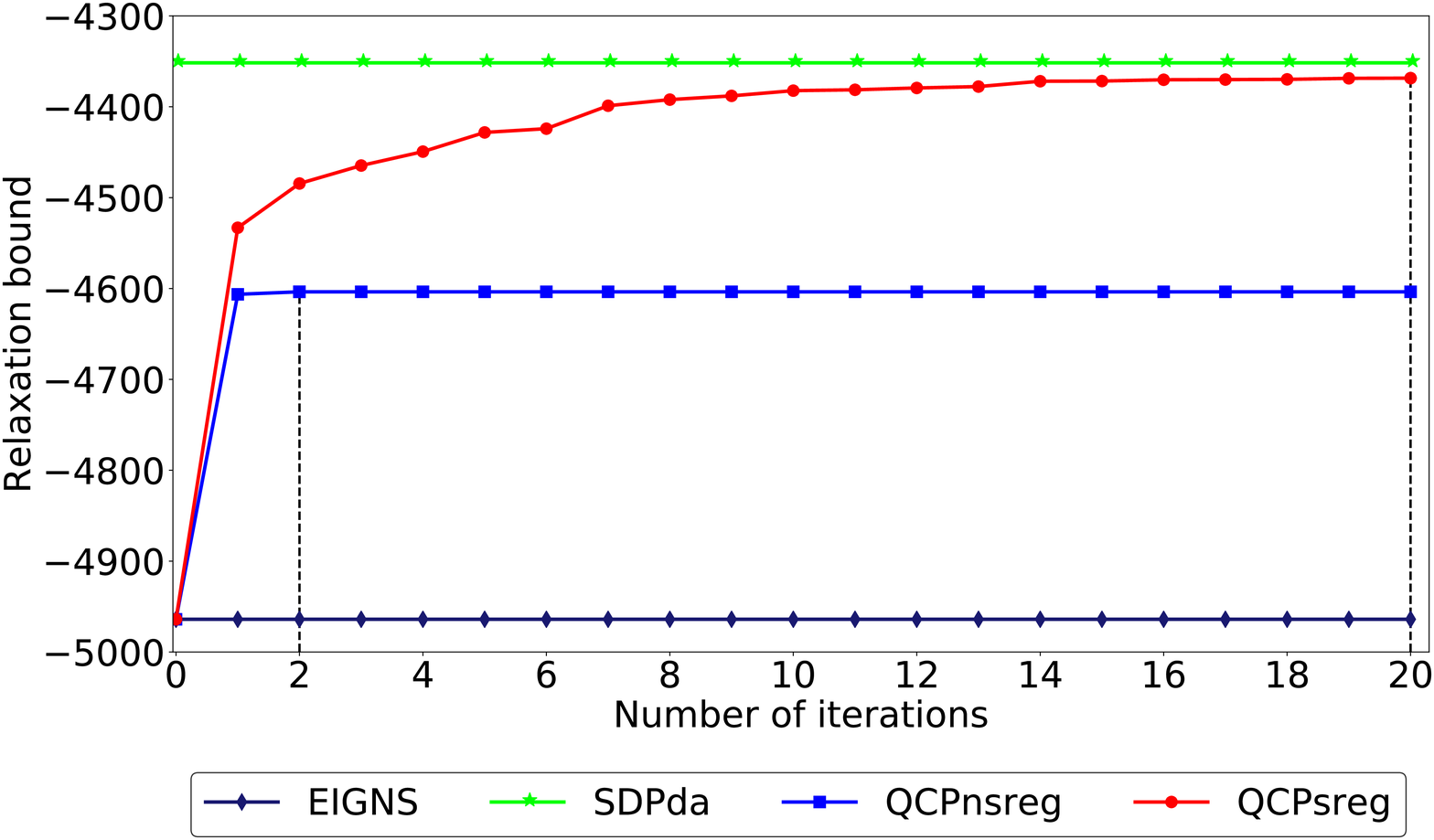}%
} \\
\subfloat[Selected BoxQP instance.\label{fig:BoxQP_example}]{%
  \includegraphics[scale=0.12]{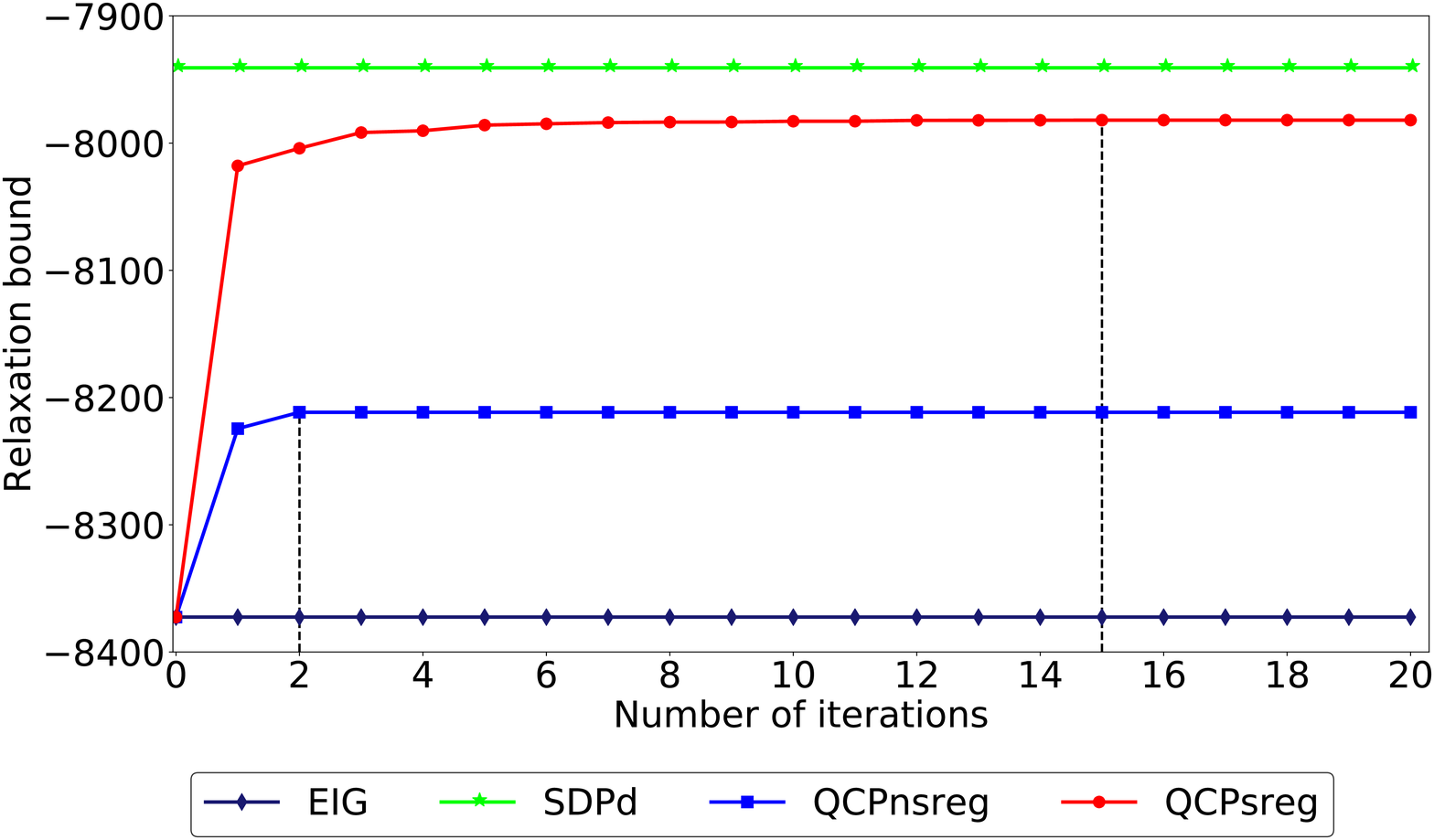}%
}
\subfloat[Selected EIQP instance.\label{fig:EIQP_example}]{%
  \includegraphics[scale=0.12]{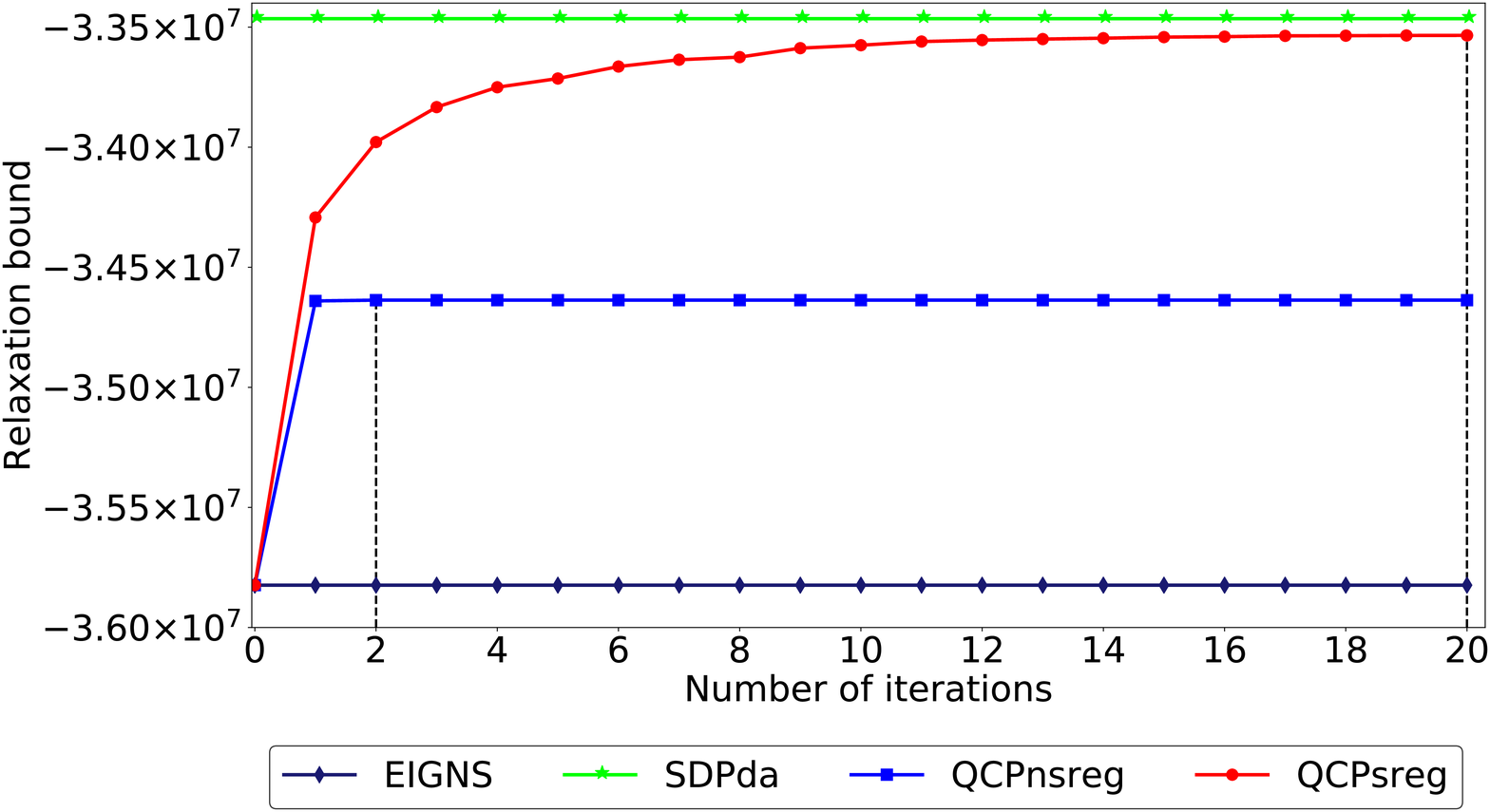}%
}
\caption{Comparison between the two versions of the cutting surface algorithm for selected test problems.}
\label{fig:root_node_gaps_selected_instances}
\end{figure}

Now, we compare the two versions of Algorithm~\ref{alg:cutting_surface_algorithm} by considering all the instances contained in each of the test libraries. To this end, we construct performance profiles based on the following root-node relaxation gap:
\begin{equation}
\label{root_node_gap}
{\text{Root gap}} = \left(\dfrac{\mu_{\text{SDP}} - \mu_{\text{QCP}}}{\mu_{\text{SDP}} - \mu_{\text{QP}}}\right) \times 100
\end{equation}
where $\mu_{\text{QCP}}$ is the lower bound given by the last QCP relaxation solved in a given version of Algorithm~\ref{alg:cutting_surface_algorithm}, and $\mu_{\text{QP}}$ and $\mu_{\text{SDP}}$ denote the lower bounds provided by the corresponding spectral and SDP relaxations. A smaller gap represents a better approximation of the corresponding SDP bound.

The performance profiles are presented in Figures~\ref{fig:CBQP_gaps}--\ref{fig:EIQP_gaps}. These profiles show the percentage of models for which the gap defined in~\eqref{root_node_gap} is below a certain threshold. Clearly, the QCP relaxations constructed via our separation procedure provide significantly smaller gaps than the QCP relaxations derived with the separation algorithm proposed in~\cite{d:16}. 

\begin{figure}[htp]
\centering
\subfloat[960 CBQP instances.\label{fig:CBQP_gaps}]{%
  \includegraphics[scale=0.14]{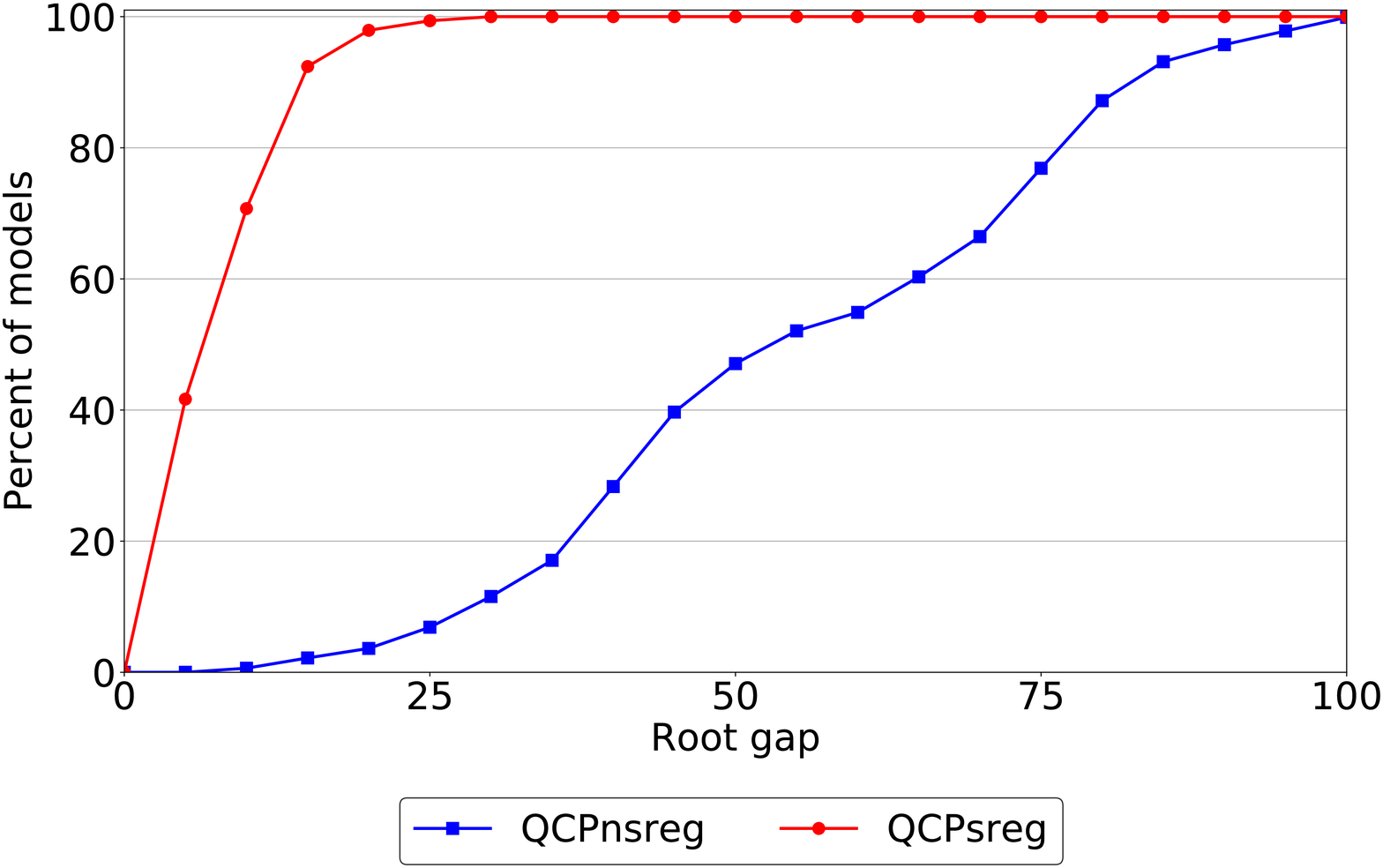}%
}
\subfloat[30 QSAP instances.\label{fig:QSAP_gaps}]{%
  \includegraphics[scale=0.14]{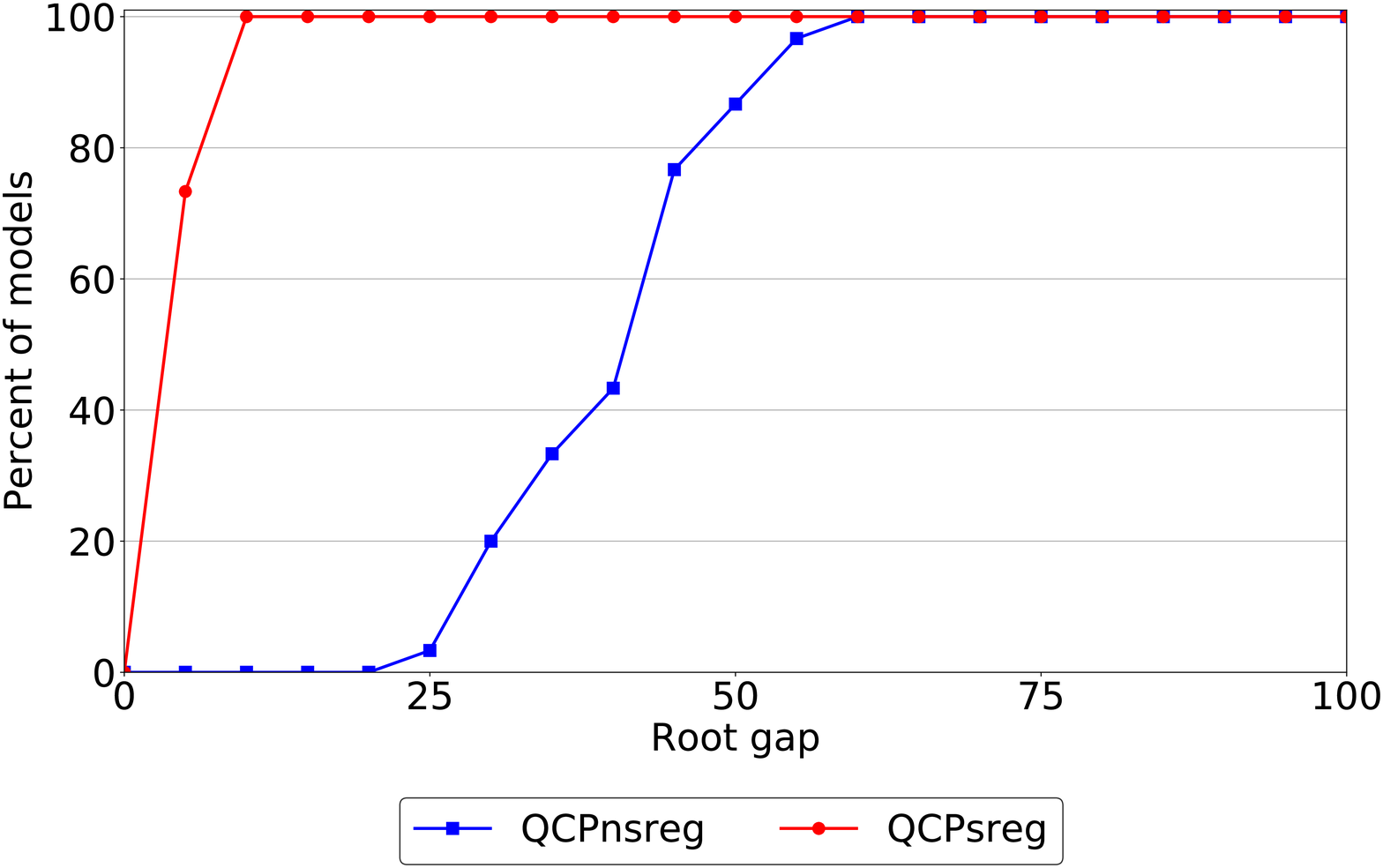}%
} \\
\subfloat[246 BoxQP instances.\label{fig:BoxQP_gaps}]{%
  \includegraphics[scale=0.14]{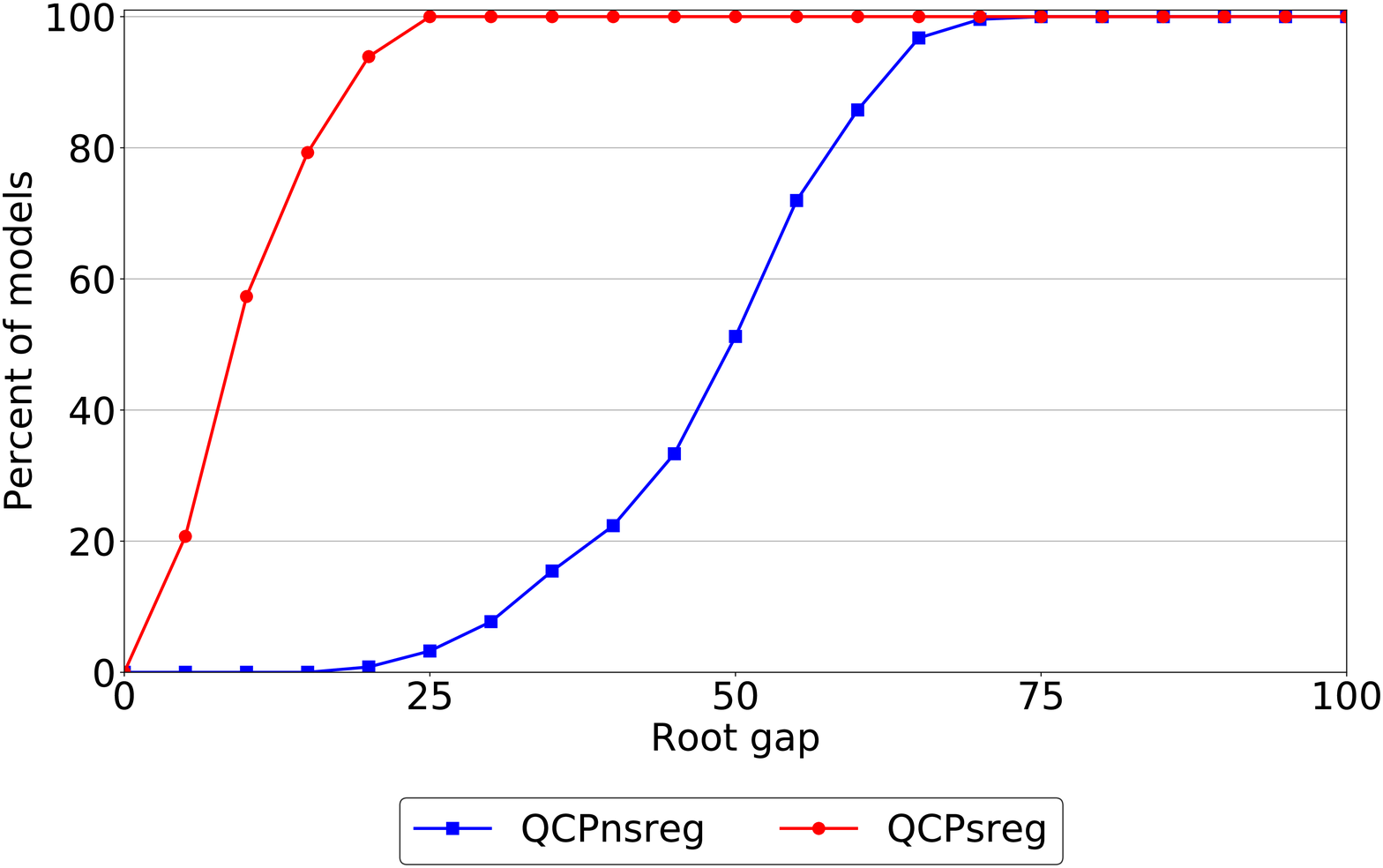}%
}
\subfloat[315 EIQP instances.\label{fig:EIQP_gaps}]{%
  \includegraphics[scale=0.14]{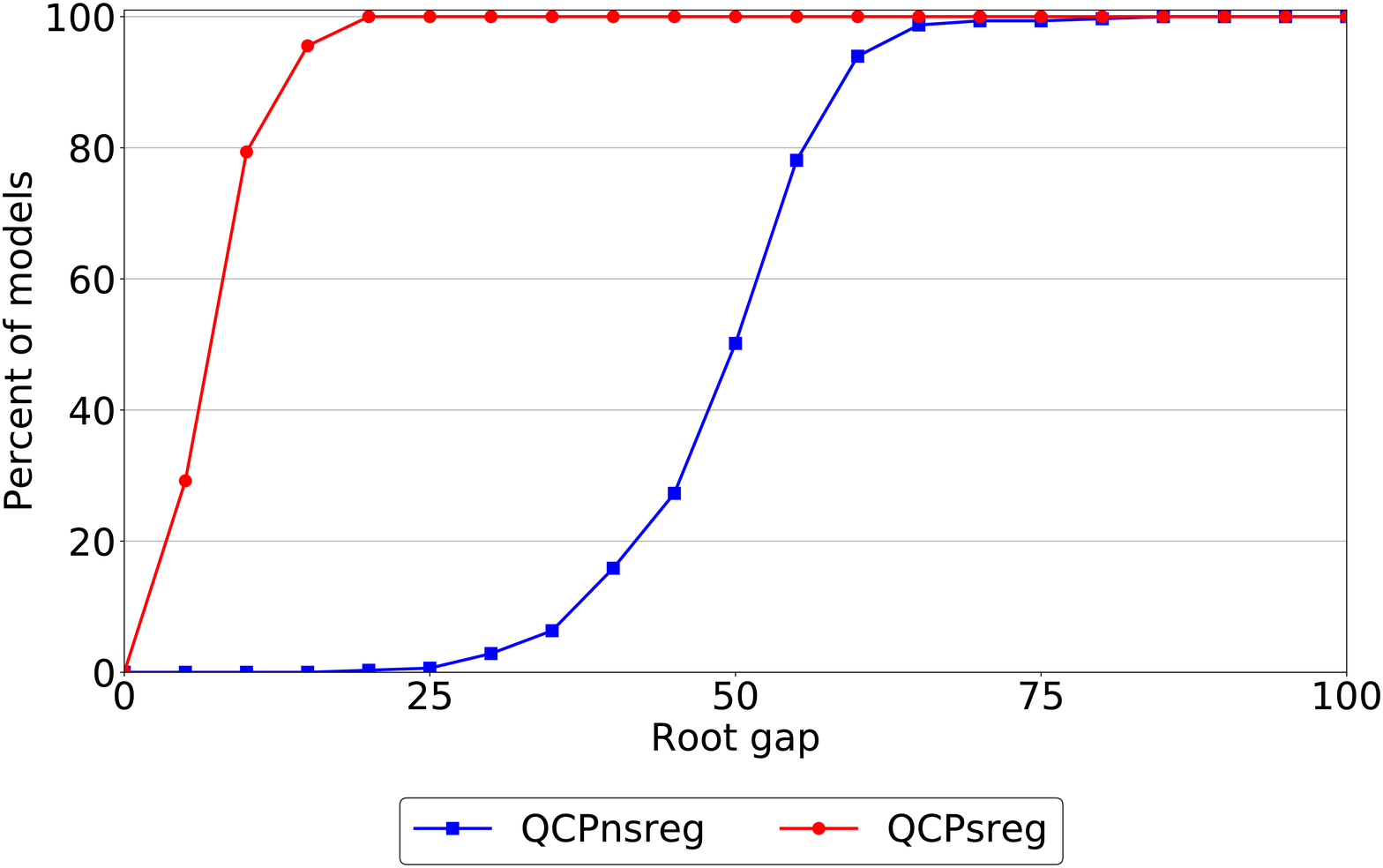}%
}
\caption{Comparison between the two versions of the cutting surface algorithm for all test problems.}
\label{fig:root_node_gaps_all_instances}
\end{figure}

\subsection{Impact of the implementation on BARON's performance}
\label{sec:results_baron}

In this section, we demonstrate the benefits the relaxations introduced in this paper on the performance of the global optimization solver BARON. In our experiments, we compare the following versions of BARON 20.4:

\begin{enumerate}[(i)]

\item BARONnoqc: Version of BARON for which we disable the quadratic relaxations proposed in this paper. Note that this version of BARON includes the spectral relaxations introduced in~\cite{nrs:20}.

\item BARON: Version of BARON which uses the quadratic relaxations proposed in this paper as described in \S\ref{implementation}.

\end{enumerate}

In this comparison, we exclude from the test set all problems for which the new quadratic relaxations are not activated by BARON during the branch-and-bound search (367 instances). We also eliminate problems that can be solved trivially by both solvers (62 instances). A problem is regarded as trivial if it can be solved by both solvers in less than one second. After eliminating all of these problems from the original test set, we obtain a new test set consisting of 1122 instances.

We first consider the nontrivial problems that are solved to global optimality by at least one of the two the versions of the solver (259 instances). For this analysis, we compare the performance of the two solvers by considering the following metrics: (i) CPU time, (ii) total number of nodes in the branch-and-bound tree (iterations), and (iii) maximum number of nodes stored in memory (memory). In this comparison, we say that the two solvers perform similarly if any of these metrics are within 10\% of each other. The results are presented in Figures~\ref{fig:baron_vs_baron_noqc_cpu_times}--\ref{fig:baron_vs_baron_noqc_memory}. As the figures indicate, for nearly 90\% of the problems considered in this comparison, our implementation leads to a significant reduction in CPU time, number of iterations, and memory requirements. There are a few instances for which BARONnoqc performs slightly better than BARON. These are relatively easy instances which can be solved to global optimality by both solvers in less than 10 seconds. For these instances, the relaxations proposed in this paper provide tighter lower bounds, but the increased computational cost associated with the construction and solution of these relaxations leads to some degradation in performance.

Now, we consider the nontrivial problems that neither of the two solvers are able to solve to global optimality within the time limit (863 instances). In this case, we analyze the performance of these solvers by comparing the relative gaps reported at termination, which are determined as:
\begin{equation}
\label{optimality_gap}
\text{Relative gap} = \left(\dfrac{\mu_{\text{UBD}} - \mu_{\text{LBD}}}{\max(|\mu_{\text{LBD}}|, 10^{-3})}\right) \times 100
\end{equation}
where $\mu_{\text{LBD}}$ and $\mu_{\text{UBD}}$ respectively denote the lower and upper bounds reported by a given solver at termination. In this comparison, we say that two solvers obtain similar gaps if their relative gaps are within 10\% of each other. The results are presented in Figure~\ref{fig:baron_vs_baron_noqc_gaps}. As seen in the figure, for more than 90\% of the problems considered in this comparison, BARON reports significantly smaller gaps than BARONnoqc.

\begin{figure}[htp]
\centering
\subfloat[CPU times (259 nontrivial instances).\label{fig:baron_vs_baron_noqc_cpu_times}]{%
  \includegraphics[scale=0.17]{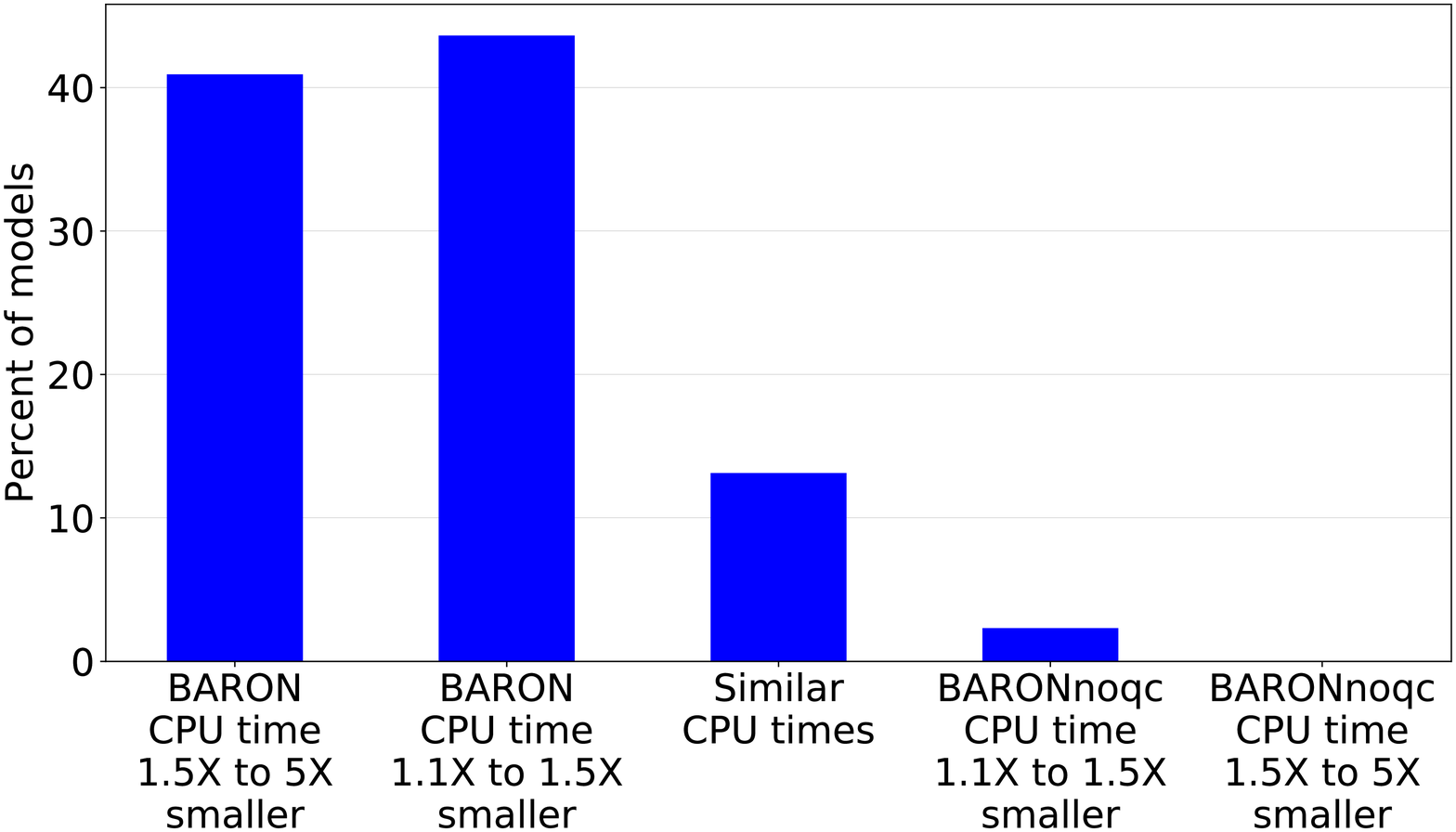}%
}
\subfloat[Iterations (259 nontrivial instances).\label{fig:baron_vs_baron_noqc_iterations}]{%
  \includegraphics[scale=0.17]{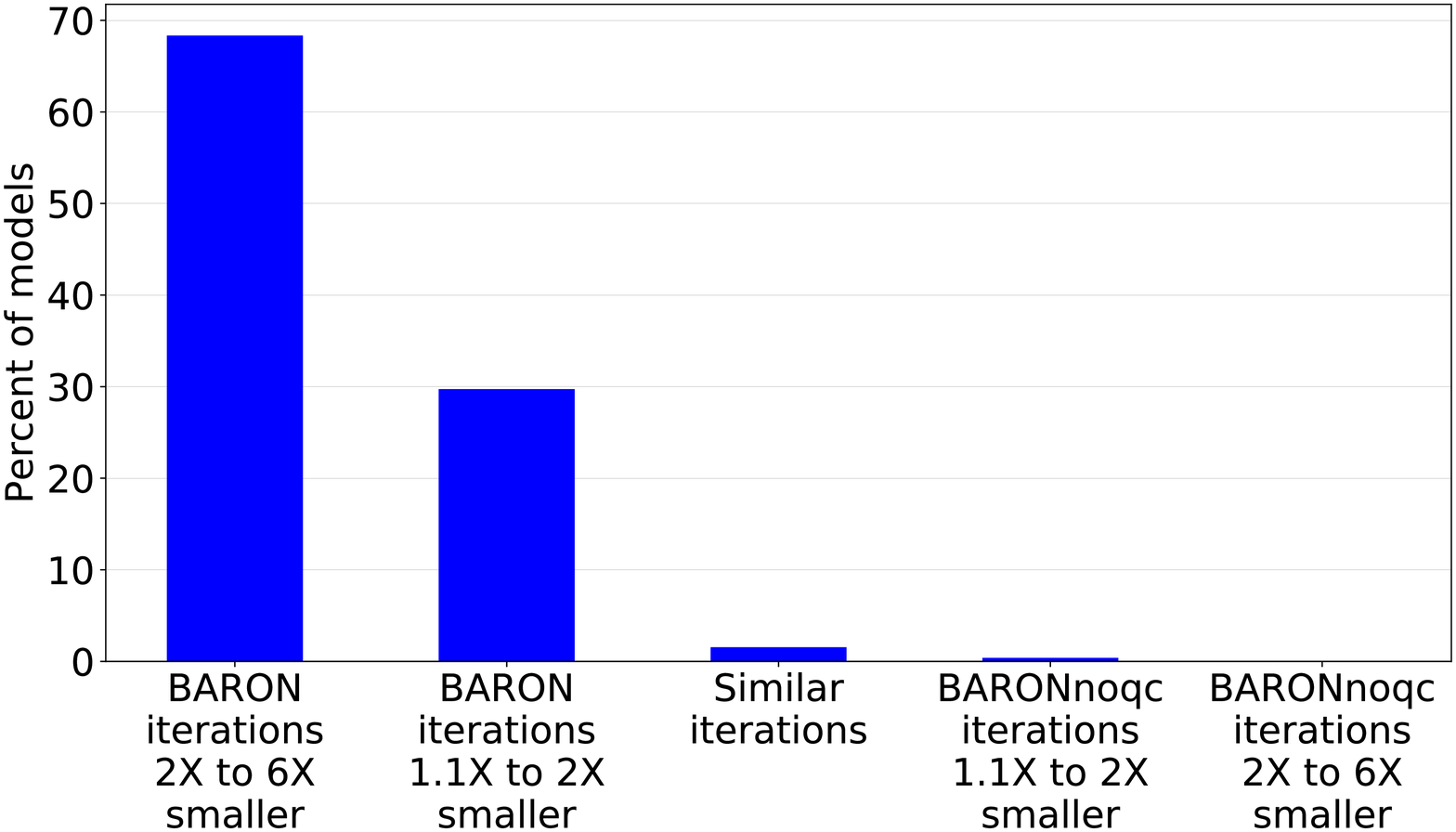}%
} \\
\subfloat[Memory (259 nontrivial instances).\label{fig:baron_vs_baron_noqc_memory}]{%
  \includegraphics[scale=0.17]{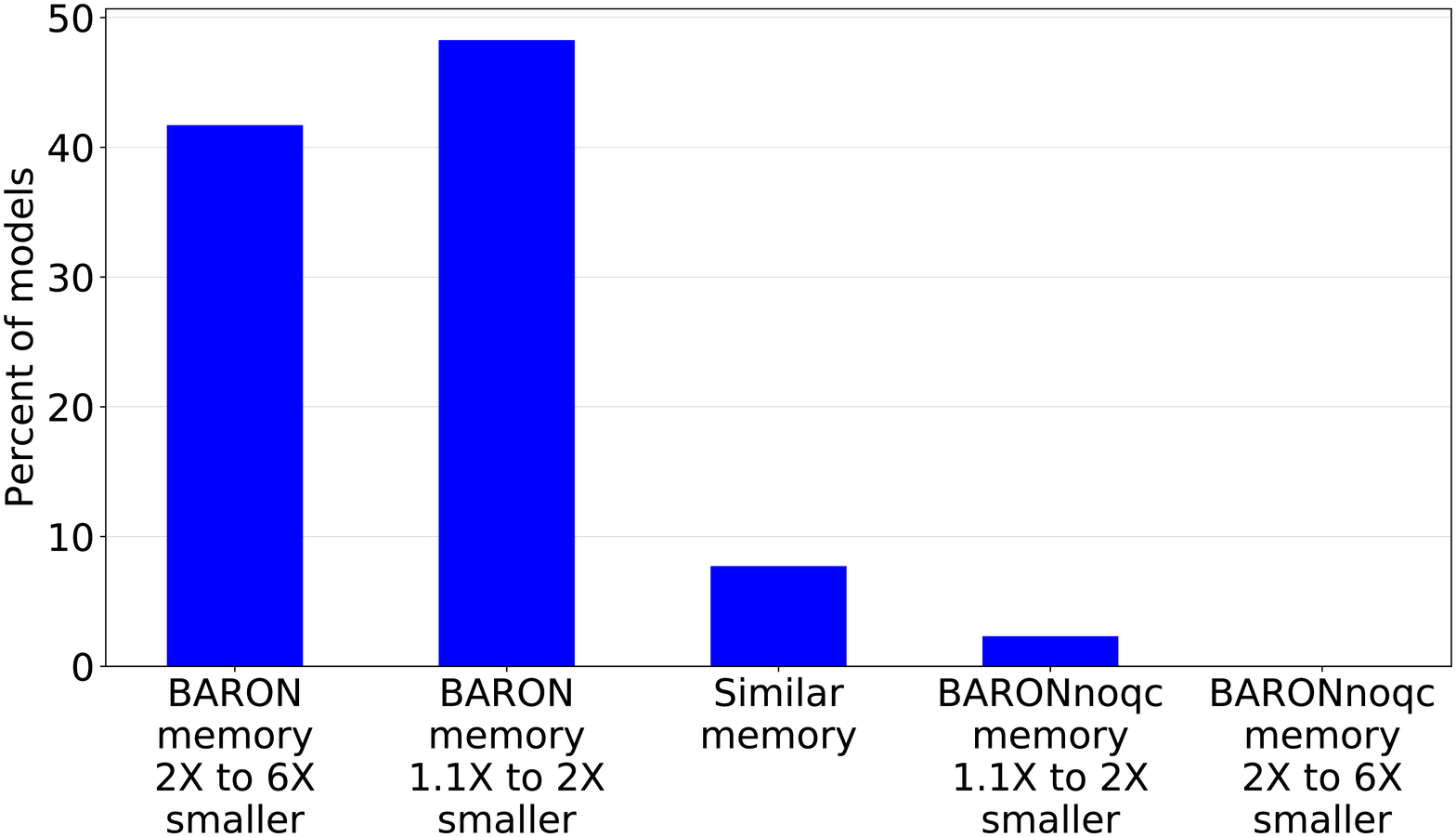}%
}
\subfloat[Relative gaps (863 nontrivial instances).\label{fig:baron_vs_baron_noqc_gaps}]{%
  \includegraphics[scale=0.17]{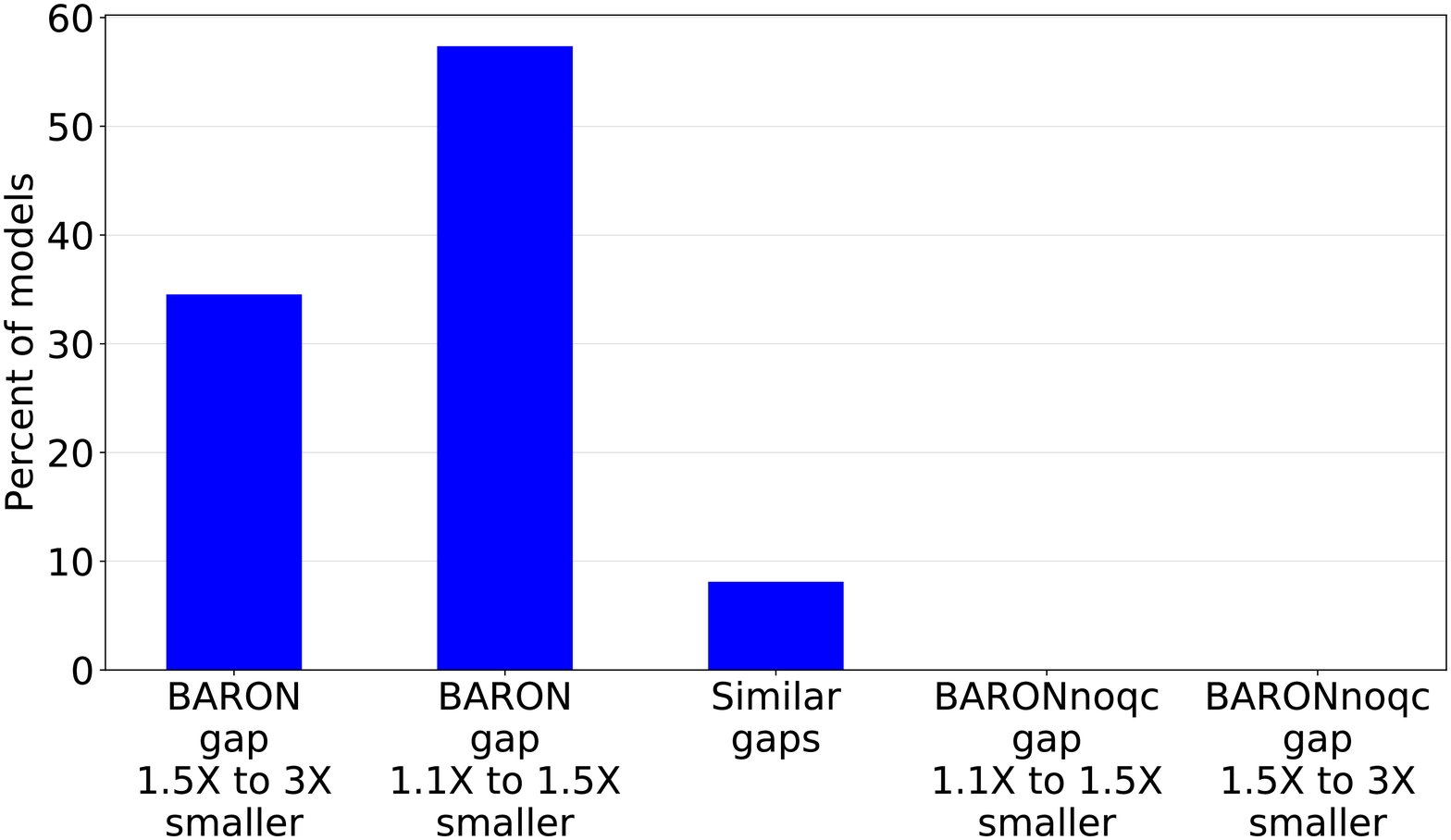}%
}
\caption{One-to-one comparison between BARON and BARONnoqc.}
\label{fig:baron_vs_baron_noqc}
\end{figure}

We finish this section by providing a more detailed analysis of the results presented in Figures~\ref{fig:baron_vs_baron_noqc_cpu_times}--\ref{fig:baron_vs_baron_noqc_gaps}. To this end, we calculate the shifted geometric means for each of the metrics considered in these figures. We use a shift factor of 1 for the CPU times and relative gaps, and a shift factor of 10 for the total number of nodes and maximum number of nodes stored in memory. The results are presented in Table~\ref{tab:baron_vs_baron_noqc}. As seen in the table, BARON significantly outperforms BARONnoqc for each of the considered metrics.

\begin{table}[htbp]
\centering
\caption{Shifted geometric means for BARON and BARONnoqc.}
\label{tab:baron_vs_baron_noqc}
\begin{tabular}{ccccc}
\toprule
Solver           & \begin{tabular}[c]{@{}c@{}}CPU Time\\ (259 instances)\end{tabular} & \begin{tabular}[c]{@{}c@{}}Iterations\\ (259 instances)\end{tabular} & \begin{tabular}[c]{@{}c@{}}Memory\\ (259 instances)\end{tabular} & \begin{tabular}[c]{@{}c@{}}Relative gaps\\ (863 instances)\end{tabular} \\
\midrule
BARONnoqc        & 14.0                                                               & 926.1                                                                & 25.9                                                             & 11.8                                                                    \\
BARON            & 9.9                                                                & 391.7                                                                & 13.5                                                             & 8.7                                                                     \\
\midrule
Improvement (\%) & 29.6                                                               & 57.7                                                                 & 47.7                                                             & 25.7                                                                   \\
\bottomrule          
\end{tabular}
\end{table}

\subsection{Comparison between global optimization solvers}
\label{sec:results_solvers}

We start this section by comparing several state-of-the-art global optimization solvers through performance profiles. For instances for which a solver can prove global optimality within the time limit, we plot the percentage of models solved within a certain amount of time. For problems for which a solver cannot prove global optimality within the time limit, we plot the percentage of models for which the relative gap defined in~\eqref{optimality_gap} is below a given threshold. These profiles are shown in Figures~\ref{fig:global_solvers_profiles_CBQP}--\ref{fig:global_solvers_profiles_EIQP}. As seen in these figures, BARON performs well relative to the other solvers. For the CBQP and QSAP instances, BARON is faster than the other solvers and solves many more problems to global optimality. For problems for which global optimality cannot be proven within the time limit, BARON terminates with smaller gaps than the other solvers.

\begin{figure}[htp]
\centering
\subfloat[960 CBQP instances. \label{fig:global_solvers_profiles_CBQP}]{%
  \includegraphics[scale=0.14]{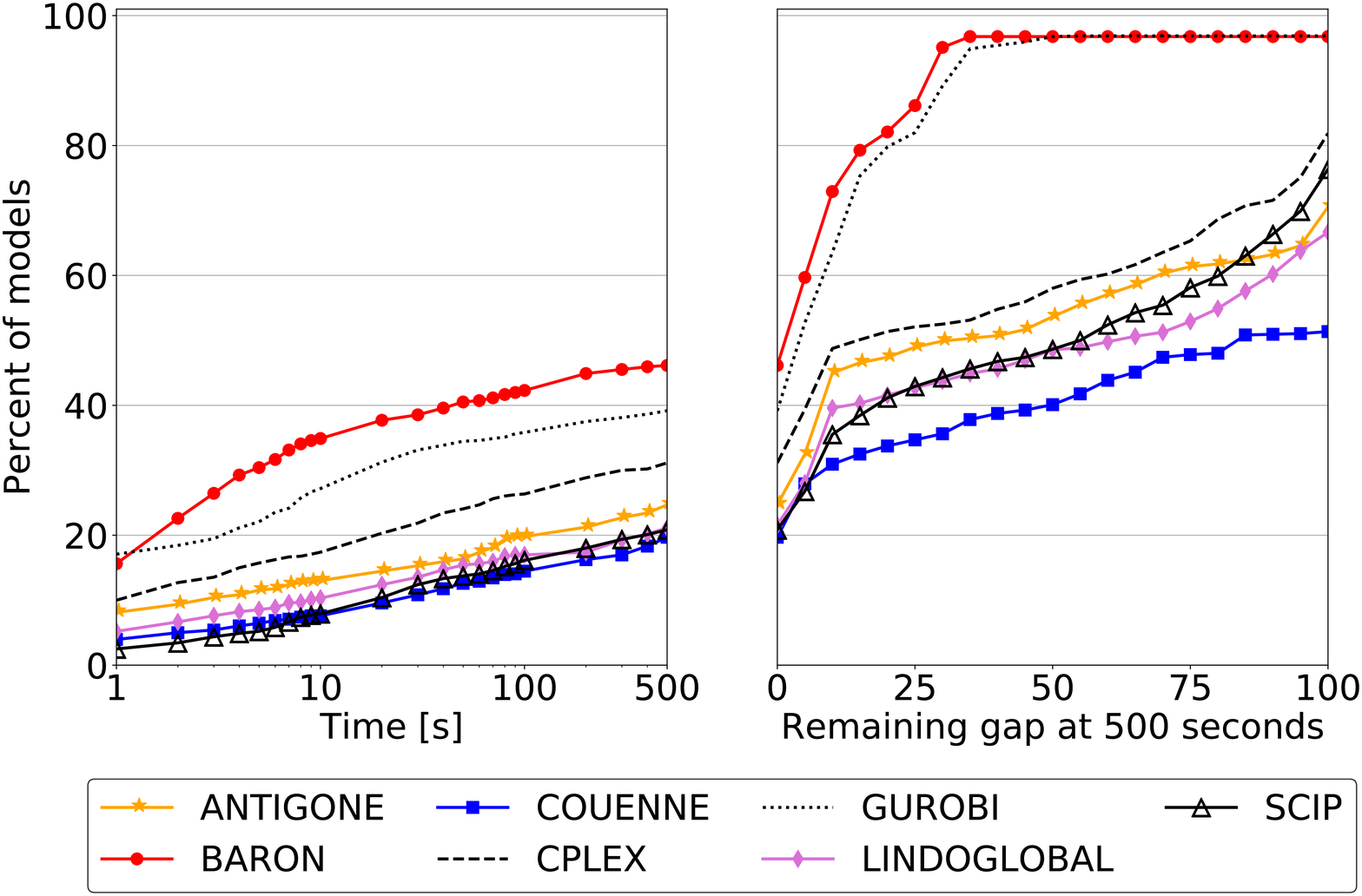}%
}
\subfloat[30 QSAP instances. \label{fig:global_solvers_profiles_QSAP}]{%
  \includegraphics[scale=0.14]{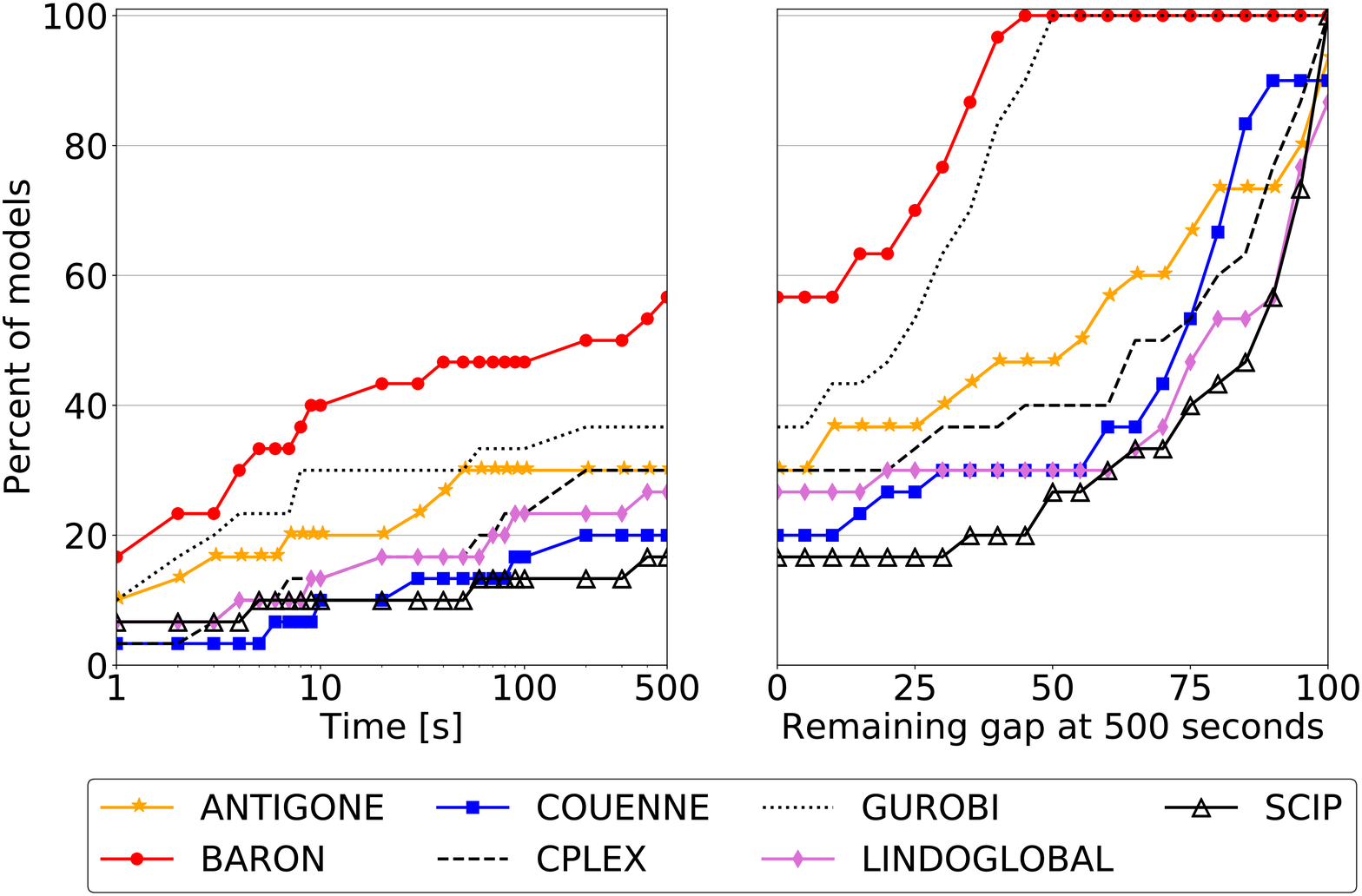}%
} \\
\subfloat[246 BoxQP instances. \label{fig:global_solvers_profiles_BoxQP}]{%
  \includegraphics[scale=0.14]{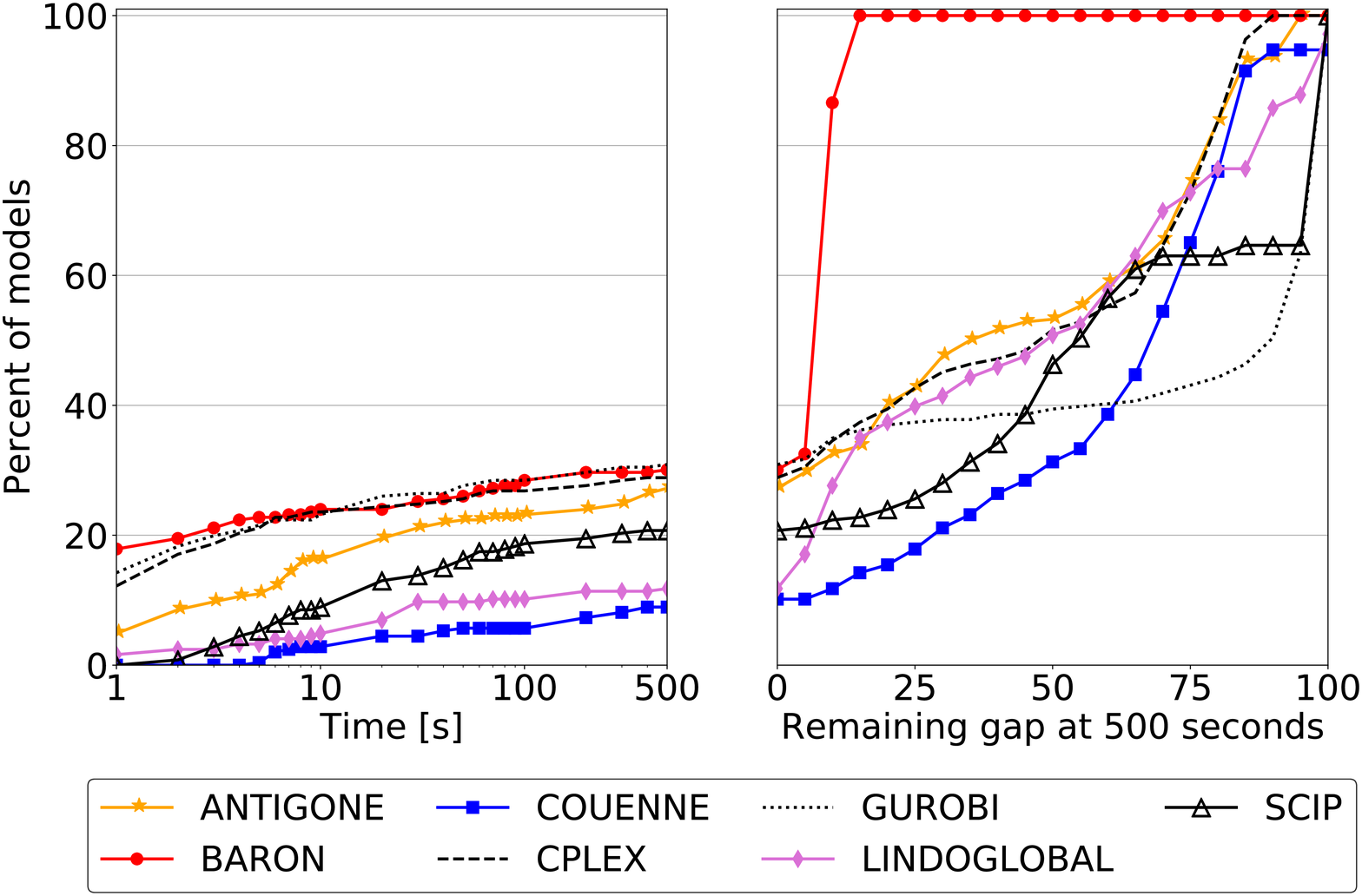}%
}
\subfloat[315 EIQP instances. \label{fig:global_solvers_profiles_EIQP}]{%
  \includegraphics[scale=0.14]{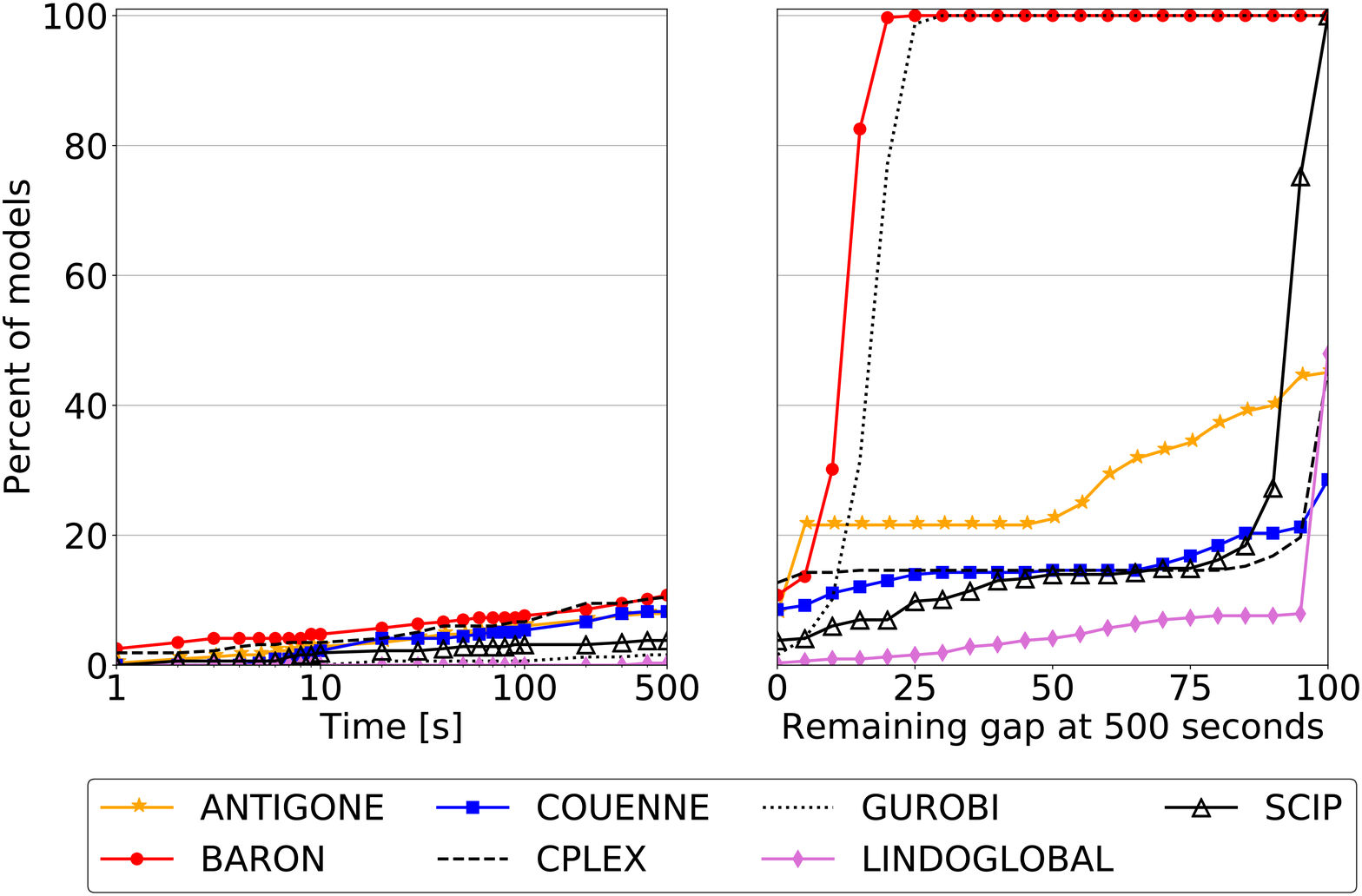}%
}
\caption{Comparison between global optimization solvers.}
\label{fig:global_solvers_profiles}
\end{figure}

Next, we provide a more detailed analysis involving BARON, CPLEX and GUROBI. We use the same type of bar plots employed in \S\ref{sec:results_baron}. For a one-to-one comparison between BARON and CPLEX, we eliminate from the test set all problems solved trivially by both solvers (124 instances), obtaining a new test set with 1427 instances. In Figure~\ref{fig:baron_vs_cplex_cpu_times}, we consider the nontrivial problems solved to global optimality by at least one of the two solvers (453 instances), whereas in Figure~\ref{fig:baron_vs_cplex_gaps}, we consider nontrivial problems that neither solver can solve to global optimality within the time limit (974 instances). As both figures indicate, BARON performs significantly better than CPLEX. For nearly 90\% of the instances considered in Figure~\ref{fig:baron_vs_cplex_cpu_times}, BARON is at least 1.1 times faster than CPLEX, whereas for more than 98\% of the instances considered in Figure~\ref{fig:baron_vs_cplex_gaps}, BARON reports significantly smaller gaps than CPLEX.

\begin{figure}[htp]
\centering
\subfloat[CPU times (453 nontrivial instances). \label{fig:baron_vs_cplex_cpu_times}]{%
  \includegraphics[scale=0.17]{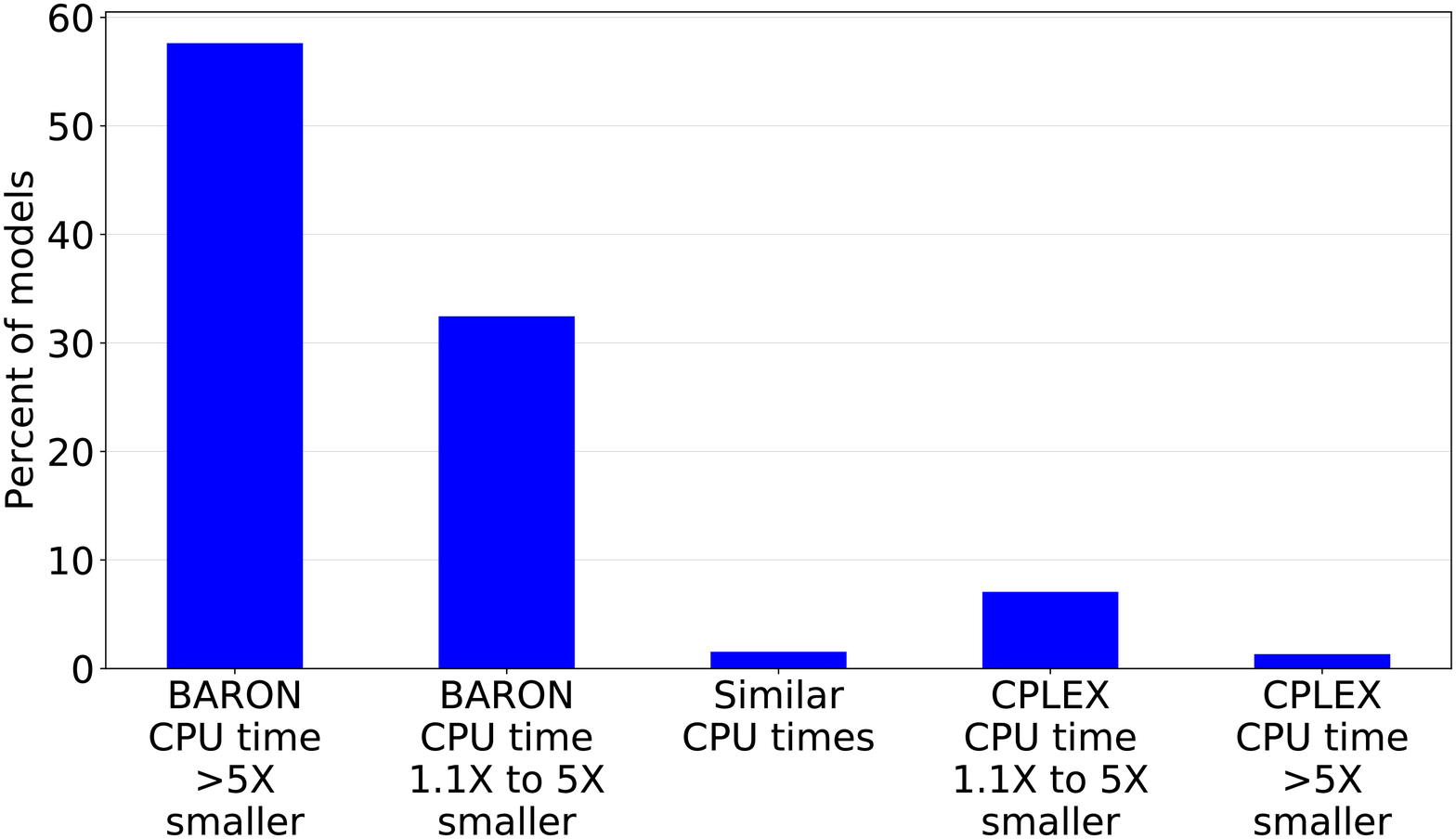}%
}
\subfloat[Relative gaps (974 nontrivial instances). \label{fig:baron_vs_cplex_gaps}]{%
  \includegraphics[scale=0.17]{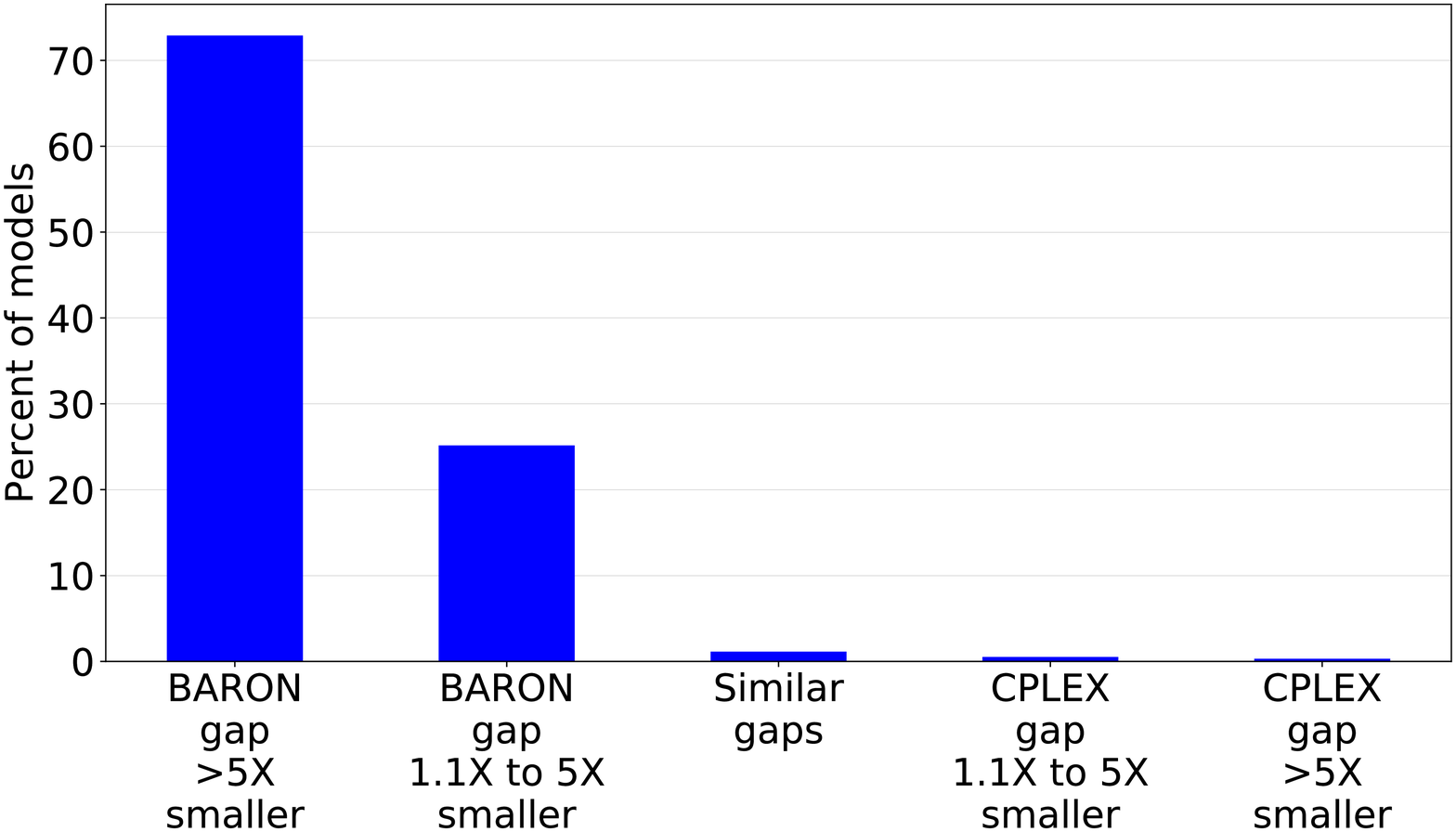}%
}
\caption{One-to-one comparison between BARON and CPLEX.}
\label{fig:baron_vs_cplex}
\end{figure}

Finally, we present a one-to-one comparison between BARON and GUROBI. Once again, we eliminate from the test set all the problems that can be solved trivially by both solvers (183 instances), which leads to a new test set with 1368 instances. In Figure~\ref{fig:baron_vs_gurobi_cpu_times}, we consider the nontrivial problems that are solved to global optimality by at least one of the two solvers (391 instances), whereas in Figure~\ref{fig:baron_vs_gurobi_gaps}, we consider nontrivial problems that neither of the two solvers are able to solve to global optimality within the time limit (977 instances). For more than 80\% of the instances considered in Figure~\ref{fig:baron_vs_gurobi_cpu_times}, BARON is at least 1.1 faster than GUROBI, whereas for nearly 90\% of the instances considered in Figure~\ref{fig:baron_vs_gurobi_gaps}, BARON terminates with considerably smaller gaps than GUROBI.

\begin{figure}[htp]
\centering
\subfloat[CPU times (391 nontrivial instances). \label{fig:baron_vs_gurobi_cpu_times}]{%
  \includegraphics[scale=0.17]{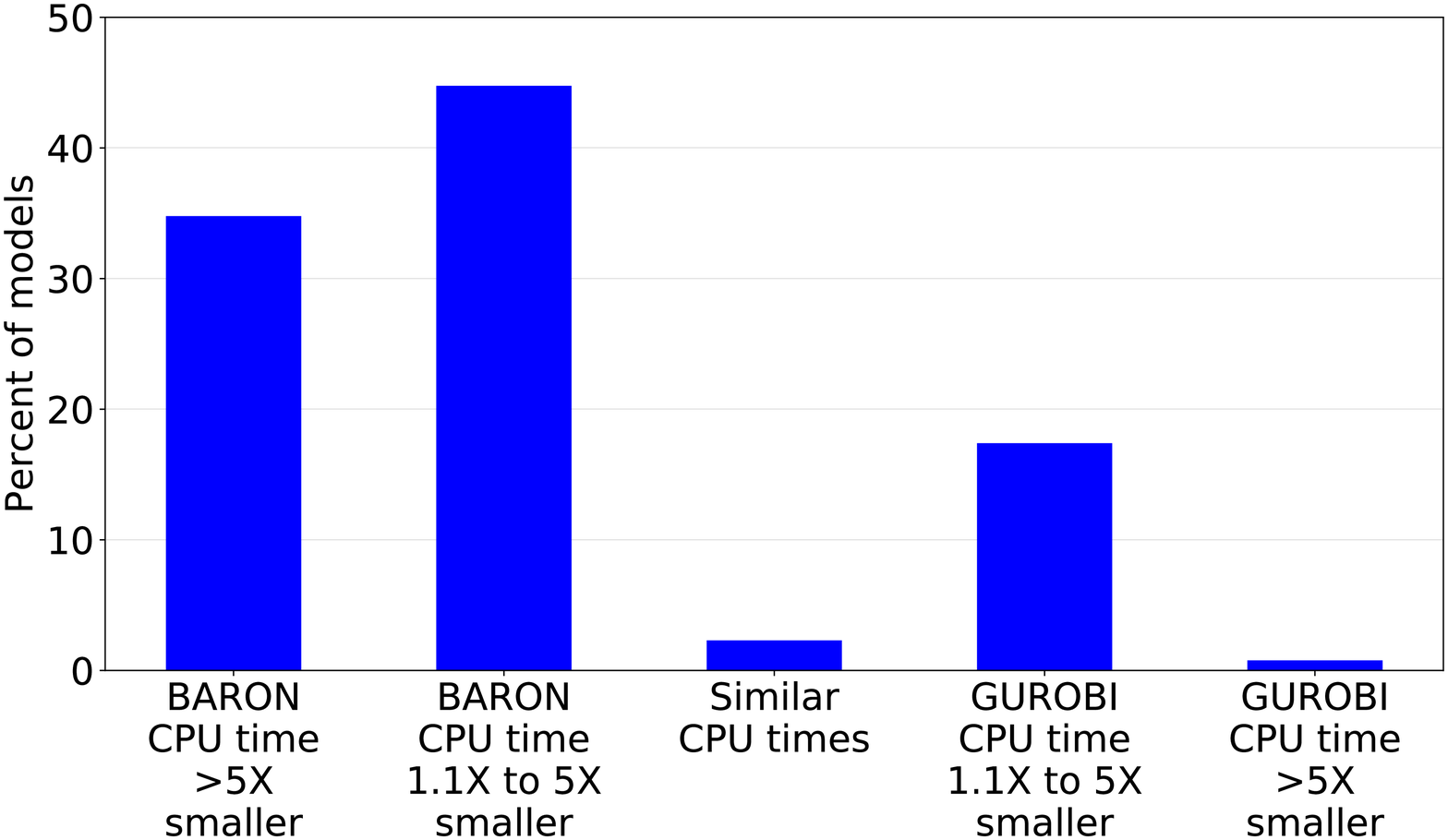}%
}
\subfloat[Relative gaps (977 nontrivial instances). \label{fig:baron_vs_gurobi_gaps}]{%
  \includegraphics[scale=0.17]{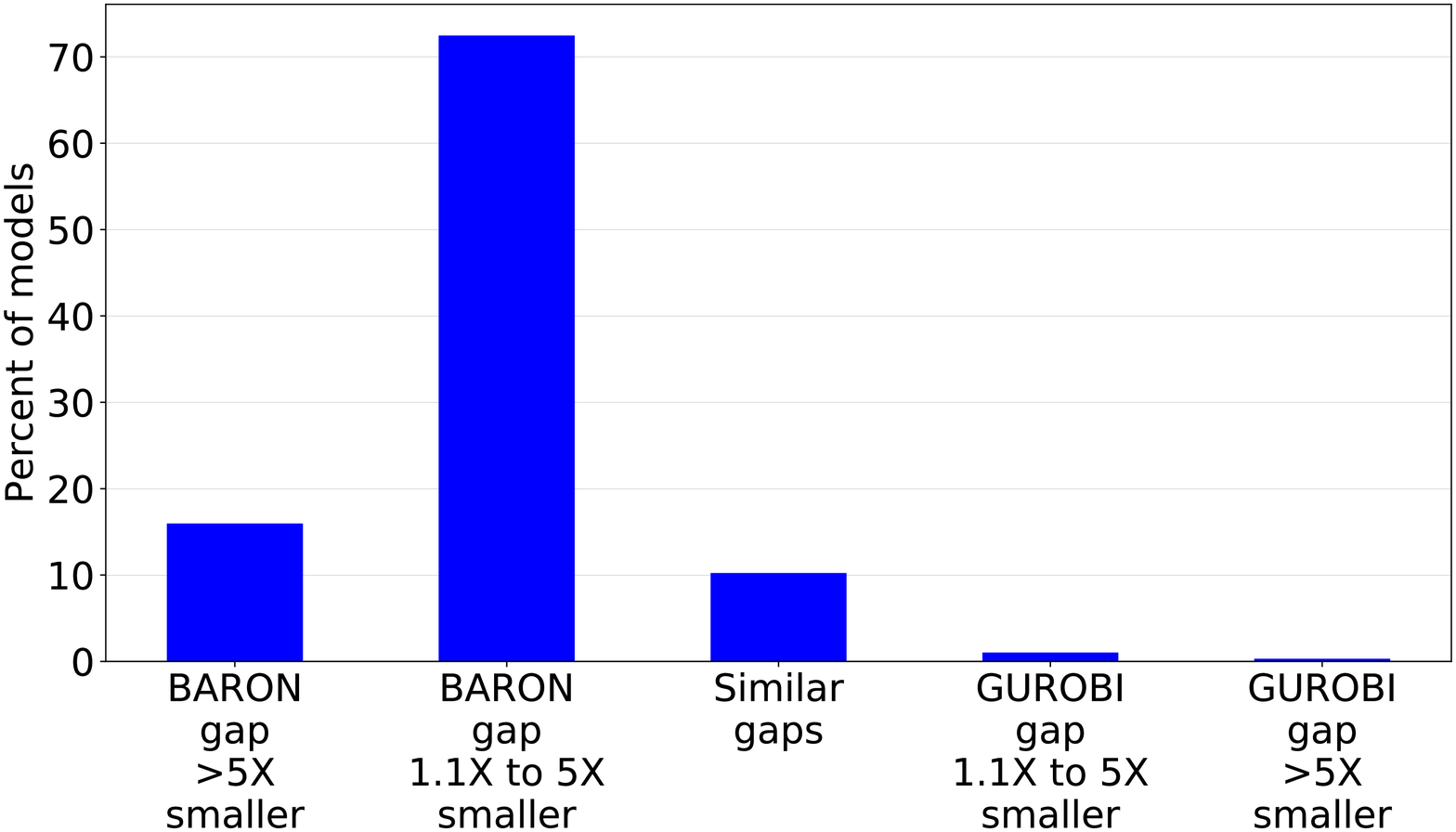}%
}
\caption{One-to-one comparison between BARON and GUROBI.}
\label{fig:baron_vs_gurobi}
\end{figure}

\section{Conclusions}
\label{conclusions}

We considered the global optimization of nonconvex MIQPs with linear constraints. We introduced a family of convex quadratic relaxations which are constructed via quadratic cuts. We investigated the theoretical properties of these relaxations and showed that they are an outer-approximation of a semi-infinite convex program which under certain conditions is equivalent to a semidefinite program. To assess the benefits of our approach, we incorporated the proposed relaxation techniques into the global optimization solver BARON, and tested our implementation on a large collection of problems. Results demonstrated that, for our test problems, our implementation leads to a very significant improvement in the performance of BARON.
\begin{appendices}
\section{Barrier  coordinate  minimization  algorithm used to solve the nonsmooth regularized separation problem}
\label{appendix}

In this appendix, we briefly describe the algorithm proposed by Dong~\cite{d:16} to solve the nonsmooth regularized separation problem~\eqref{separation_da_nonsmooth_reg}. We then describe our implementation of this algorithm. The algorithm operates on the following penalized log-det problem
\begin{equation}
\label{separation_da_nonsmooth_reg_log_det}
\begin{array}{cl}
\underset{d \in \R^n}{\text{inf}}\;\; & h(d; \sigma) := \sum\limits_{i = 1}^{n} r_i(d_i) - \sigma \text{log-det}\left(Q + \text{diag}(d) + \alpha A^T A \right) \\
  \st\;\;      & Q + \text{diag}(d) + \alpha A^T A \succ 0
\end{array}
\end{equation}
where $r_i(d_i) = \beta_i d_i$ for $d_i > 0$, $r_i(d_i) = \eta_i d_i$ for $d_i \leq 0$, $\beta_i = \eta_i + \lambda, \; \forall i \in [n]$, and $\sigma >0$. Each iteration of this algorithm involves the update of a feasible vector $\bar{d}$ and an inverse matrix $V:= {\left[ Q + \text{diag}(\bar{d}) + \alpha A^T A \right]}^{-1}$. Based on the optimality condition for~\eqref{separation_da_nonsmooth_reg_log_det}, this algorithm performs coordinate minimization by choosing an index $i$ determined as:
\begin{equation}
\label{index_choice_nonsmooth}
\begin{array}{cl}
i = \text{arg}\underset{j = 1, \dots, n}{\text{max}} \left\lbrace \left| {s(\bar{d})}_j \right| \right\rbrace, \;\; \text{with} \;\; s(\bar{d}) = \text{arg}\underset{u \in \R^n}{\text{min}} \left\lbrace {\Vert u \Vert}_2 \, : \, u \in \partial h(\bar{d}; \sigma) \right\rbrace.
\end{array}
\end{equation}
where $\partial h(\bar{d}; \sigma)$ is the subdifferential of $h(\bar{d}; \sigma)$. This choice of $i$ leads to a one-dimensional minimization problem 
similar to~\eqref{one_dimensional_problem} but involving $h(\bar{d} + \Delta d_i e_i; \sigma)$. This problem can be solved analytically to obtain the following formula for $\Delta d_i^*$ (see Section 4 in~\cite{d:16} for details):
\begin{equation}
\label{delta_di_nonsmooth}
\Delta d_i^* =
\left\{
\begin{array}{ll}
\frac{\sigma}{\beta_i} - \frac{1}{V_{ii}},\;
&{\rm if}\; -\bar{d}_i < - \frac{1}{V_{ii}} \;\;\; \text{or} \;\;\; \sigma \frac{V_{ii}}{1 - \bar{d}_i V_{ii}} > \beta_i,\\
-\bar{d}_i,\;
&{\rm if}\; -\bar{d}_i \geq - \frac{1}{V_{ii}} \;\;\; \text{and} \;\;\; \eta_i \leq \sigma \frac{V_{ii}}{1 - \bar{d}_i V_{ii}} \leq \beta_i,\\
\frac{\sigma}{\eta_i} - \frac{1}{V_{ii}},\;
&{\rm if}\; -\bar{d}_i \geq - \frac{1}{V_{ii}} \;\;\; \text{and} \;\;\; \sigma \frac{V_{ii}}{1 - \bar{d}_i V_{ii}} < \eta_i.
\end{array}\right.
\end{equation}
After calculating $\Delta d_i^*$ according to~\eqref{delta_di_nonsmooth}, $\bar{d}$ and $V$ are updated using~\eqref{update_d} and~\eqref{update_V}, respectively. Once~\eqref{separation_da_nonsmooth_reg_log_det} has been solved within a given precision, the penalty parameter $\sigma$ is adjusted through a rule similar to~\eqref{update_sigma}:
\begin{equation}
\label{update_sigma_nonsmooth}
\begin{aligned}
\sigma \leftarrow \max\{\sigma_{\text{min}}, \sigma_{\text{upd}} \cdot \sigma \} \;\;\;\; \text{if} \;\;\;\; \dfrac{ {\Vert s(\bar{d}) \Vert}_2 }{ {\Vert \beta \Vert}_2 } \leq \epsilon_{\text{upd}}
\end{aligned}
\end{equation}
where $s(\bar{d})$ is used as a measure of optimality. The relative improvement in the objective function of~\eqref{separation_da_nonsmooth_reg} is checked every $\omega_{\text{check}}  n$ iterations, and the algorithm terminates if this relative improvement is smaller than $\epsilon_{\text{check}}$.  

The entire procedure is summarized in Algorithm~\ref{alg:coordinate_minimization_nonsmooth}. In our implementation of this algorithm, we use an initial perturbation $\hat{d} = -1.5 \lambda_{\text{min}} (Q + \alpha A^T A)  \mathbbm{1}$. We set the following parameters by using the values recommended in~\cite{d:16}: $\lambda = \sum_{i = 1}^{n} \eta_i$, $\sigma_{\text{min}} = 10^{-5}$, $\sigma_{\text{upd}} = 0.8$, $\epsilon_{\text{upd}} = 0.03$. We use $\text{MaxIter} = 500 n$, $\omega_{\text{check}} = 10$ and $\epsilon_{\text{check}} = 10^{-4}$. The initial value of $\sigma_{\text{init}}$ is determined as:
\begin{equation}
\label{sigma_formula_nonsmooth}
\begin{aligned}
\sigma_{\text{init}} = {{\text{median}} \left\lbrace \left| \dfrac{u_i}{V_{ii}} \right| \right\rbrace}_{i = 1}^{n}
\end{aligned}
\end{equation}
where $ u_i \in \partial r_i(\hat{d}_i)$, and $\partial r_i(\hat{d}_i)$ is the subdifferential of $r_i(\hat{d}_i)$.

\begin{algorithm}
\caption{Barrier coordinate minimization algorithm used to solve the nonsmooth regularized separation problem~\eqref{separation_da_nonsmooth_reg}}
\label{alg:coordinate_minimization_nonsmooth}
\begin{algorithmic}[1]
\STATE \textbf{Input}: $Q$, $A$, $\alpha$, and an optimal solution $(\bar{x}, \bar{y}, \bar{v})$ to~\eqref{qcp_das}.
\STATE \textbf{Output}: A vector $\bar{d}$ that solves~\eqref{separation_da_nonsmooth_reg}.
\STATE Set $\lambda = \sum_{i = 1}^{n} \eta_i$.
\STATE Set $\bar{d} = \hat{d}$, where $\hat{d} = 1.5 \mu \mathbbm{1}$ and $\mu = -\lambda_{\text{min}} (Q + \alpha A^T A)$.
\STATE Set $\sigma = \sigma_{\text{init}}$, where $\sigma_{\text{init}}$ is calculated according to~\eqref{sigma_formula_nonsmooth}.
\STATE Calculate $V = {\left[ Q + \text{diag}(\bar{d}) + \alpha A^T A \right]}^{-1}$ and set $k = 0$.
\WHILE{$(k < \text{MaxIter})$}
\STATE Update $k \leftarrow k + 1$.
\STATE Determine an index $i$ according to~\eqref{index_choice_nonsmooth} and calculate $\Delta d_i^*$ using~\eqref{delta_di_nonsmooth}.
\STATE Update $\bar{d}$ according to~\eqref{update_d}.
\STATE Update $V$ according to~\eqref{update_V}.
\STATE Adjust $\sigma$ according to~\eqref{update_sigma_nonsmooth}.
\IF{$(k \bmod (\omega_{\text{check}} n) = 0)$}
\STATE Terminate if the improvement in the objective of~\eqref{separation_da_nonsmooth_reg} is smaller than $\epsilon_{\text{check}}$.
\ENDIF
\ENDWHILE
\end{algorithmic}
\end{algorithm}

\end{appendices}

\bibliographystyle{spmpsci}      
\bibliography{carlos,reference}   

%
%

\end{document}